[1994/12/01]
\documentclass[final]{aomart}

\usepackage{ifthen,latexsym,amssymb,amsmath,bbm,fixmath,graphicx}

\newcommand{\C}[1]{\mathcal{#1}}
\newcommand{\B}[1]{{\bf #1}}
\newcommand{\I}[1]{{\mathbbm #1}}

\newcommand{\M}[1]{\mathrm{#1}}

\newcommand{\e}{\varepsilon}

\renewcommand{\mid}{:}
\renewcommand{\dots}{\hspace{0.9pt}.\hspace{0.3pt}.\hspace{0.3pt}.\hspace{1.5pt}}
\renewcommand{\ge}{\geqslant}
\renewcommand{\le}{\leqslant}

\newcommand{\complicated}{rich}

\newif\ifnotesw\noteswtrue
\newcommand{\comment}[1]{\ifnotesw $\blacktriangleright$\ {\sf #1}\   $\blacktriangleleft$ \fi}
\noteswfalse	
\newcommand{\hide}[1]{}

\newcommand{\beq}[1]{\begin{equation}\label{#1}}
\newcommand{\eeq}{\end{equation}}

\newtheorem{theorem}{Theorem}[section]
\newtheorem{lemma}[theorem]{Lemma}
\newtheorem{conjecture}[theorem]{Conjecture}

\theoremstyle{definition}

\newtheorem{rem}[theorem]{Remark}

\newcommand{\case}[1]{\medskip\noindent{\bf Case #1} }
\newtheorem{cla}{Claim}[theorem]
\newcommand{\claim}[2]{\begin{cla}\label{cl:#1}#2\end{cla}}

\newcommand{\bcpf}{\begin{proof}[Proof of Claim.]}
\newcommand{\ecpf}{\end{proof}}
\newcommand{\secref}[1]{Section~\ref{#1}}

\newcommand{\V}[1]{\mathbold{#1}}
\newcommand{\eps}{\varepsilon} 
\newcommand{\MD}{M} 
\renewcommand{\mid}{:}
\newcommand{\dist}{\M{dist}}
\newcommand{\simTr}{\overset{{\mbox{\tiny Tr}}}{\sim}}

\newcommand{\CA}{\mathsf{A}}
\newcommand{\CB}{\mathsf{B}}
\newcommand{\CC}{{C}}
\newcommand{\CF}{{F}}
\newcommand{\CO}{{\C O}}
\newcommand{\VV}[2]{{\mathcal V}_{#2}}

\newcommand{\CN}{{\C N}}
\newcommand{\CP}{{\C P}}
\renewcommand{\Delta}{\mathrm{dim_{\raisebox{-3pt}{$\Box$}}}}

\title[Measurable circle squaring]{Measurable circle squaring}

\author[{\L}.~Grabowski]{{\L}ukasz Grabowski}
\address{Department of Mathematics and Statistics,
Lancaster University, Lancaster, United Kingdom}
\fulladdress{Department of Mathematics and Statistics,
Lancaster University, Lancaster \mbox{LA1 4YW}, United Kingdom}
\email{graboluk@gmail.com}

\author[A.~M\'ath\'e]{Andr\'as M\'ath\'e}
\address{Mathematics Institute, University of Warwick, Coventry, United Kingdom}
\fulladdress{Mathematics Institute, University of Warwick, Coventry \mbox{CV4 7AL}, United Kingdom}
\email{A.Mathe@warwick.ac.uk}

\author[O.~Pikhurko]{Oleg Pikhurko} 
\address{Mathematics Institute and DIMAP, University of Warwick, Coventry, United Kingdom}
\fulladdress{Mathematics Institute and DIMAP, University of Warwick, Coventry \mbox{CV4 7AL}, United Kingdom}
\email{O.Pikhurko@warwick.ac.uk}

\thanks{\L.G.\ was partially supported by EPSRC grant~EP/K012045/1 and by Fondations Sciences Math{\'e}matiques de Paris during the programme \textit{Marches Al\'eatoires et G\'eom\'etrie Asymptotique des Groupes} at Institut Henri-Poincar\'e.
A.M.\ was partially supported by a Levehulme Trust 
Early Career Fellowship and by the Hungarian National Research, Development and 
Innovation Office -- NKFIH, 104178.
O.P.\ was partially supported by ERC
grant~306493 and EPSRC grant~EP/K012045/1.
}

\copyrightyear{}
\copyrightnote{}

\subject{primary}{matsc2010}{03E05, 28A05}
\subject{secondary}{matsc2010}{03E15, 05C70, 11K36, 28A75, 28E15}

\received{\formatdate{2015-12-03}}
\revised{\formatdate{2016-08-02}}
\accepted{\formatdate{2016-08-24}}
\published{}
\publishedonline{}

\proposed{}
\seconded{}
\corresponding{}
\version{}

\begin{document}

\begin{abstract} Laczkovich proved that
if bounded subsets $A$ and $B$ of $\I R^k$ have the same non-zero 
Lebesgue measure and the upper box dimension of the boundary of each set is less than 
$k$, then there is a partition
of $A$ into finitely many parts that can be translated to form 
a partition of $B$.
Here we show that it can be additionally required that each part is both 
Baire and Lebesgue measurable. 
As special cases, this gives measurable and translation-only 
versions of Tarski's circle squaring and Hilbert's third problem.
\end{abstract}

\maketitle

\section{Introduction}\label{Intro}

We call two sets $A,B\subseteq \I R^k$ \emph{equidecomposable} and denote this as $A\sim B$ if there are a partition $A=A_1\cup\dots\cup A_n$ (into finitely many parts) and isometries $\gamma_1,\dots,\gamma_n$
of $\I R^k$ such that the images of the parts $\gamma_1(A_1),\dots,\gamma_n(A_n)$ partition~$B$. In other words, we can cut $A$ into finitely many
pieces and rearrange them to form the set $B$. When this can be done is a very basic question that one can ask about
two sets and, as Dubins, Hirsch, and Karush \cite[Page 
239]{dubins+hirsch+karush:63} write, ``\emph{variants of the problems studied 
here already occur in Euclid}". We refer the reader to various surveys and 
expositions of this area  
(\cite{gardner:91,gardner+wagon:89,hertel+richter:03,laczkovich:93:rimut,
laczkovich:94,laczkovich:02,wagon:81}) as well as the excellent book by 
Tomkowicz and Wagon~\cite{tomkowicz+wagon:btp}.

The version that is closest to our everyday intuition (e.g.\ via puzzles like ``Tangram'', ``Pentomino'', or ``Eternity'') 
is perhaps the 
\emph{dissection congruence} in $\I R^2$ where the pieces have to be polygonal 
and their boundary can be ignored when taking partitions. A well-known
example from elementary mathematics is finding  the area of a triangle by dissecting it into a rectangle. In fact, as it was discovered around 1832 independently by Bolyai and Gerwien, any  two polygons of the same area are congruent by dissections. (Apparently, Wallace proved this result already in 1807; see~\cite[Pages 34--35]{tomkowicz+wagon:btp} for a historical account and further references.) The equidecomposition problem for polygons is also completely resolved:  a result of Tarski~\cite{tarski:24} (see e.g.\ \cite[Theorem~3.9]{tomkowicz+wagon:btp}) gives that any two polygons of the same area are equidecomposable. 

Banach and Tarski~\cite{banach+tarski:24} proved that, in dimensions $3$ or higher, any two bounded sets with
non-empty interior are equidecomposable; in particular, we get the famous Banach-Tarski Paradox that a ball
can be doubled. On the other hand, as it was also shown in~\cite{banach+tarski:24} by using the earlier results of Banach \cite{banach:23}, a ball in $\I R^k$ cannot be doubled  for $k=1,2$. This prompted von Neumann~\cite{neumann:29} to investigate  what makes the cases $k=1,2$ different 
using the group-theoretic point of view, which started the study of amenable groups. 

Around that time, Tarski~\cite{tarski:25} asked if the disk and square in $\I R^2$ of the same area 
are equidecomposable, which became known as \emph{Tarski's circle squaring}. 
Von Neumann~\cite{neumann:29} showed that circle squaring is possible if arbitrary 
measure-preserving affine transformations are allowed. On the other hand, some negative evidence
was provided by Dubins, Hirsch, and Karush \cite{dubins+hirsch+karush:63} who 
showed that
a circle and a square are not \emph{scissor congruent} (when
the pieces are restricted to be topological disks and their boundary can be ignored) and by  Gardner \cite{gardner:85} who proved that circle-squaring is
impossible if we use a locally discrete subgroup of isometries of $\I R^2$. However,  the deep paper of
Laczkovich~\cite{laczkovich:90} showed that the answer to Tarski's question is affirmative. In fact, his main result (coming from the papers~\cite{laczkovich:90,laczkovich:92b,laczkovich:92}) is much more general and stronger. 
In order to state it, we need some definitions.

We call two sets $A,B\subseteq \I R^k$ \emph{equivalent} (and denote this by 
$A\simTr B$) if they are equidecomposable using translations, that is, there are 
partitions $A=A_1\cup\dots\cup A_m$ and
$B=B_1\cup\dots\cup B_m$, and vectors $\V v_1,\dots,\V v_m\in\I R^k$ such that 
$B_i=A_i+\V v_i$ for each $i\in\{1,\dots,m\}$.
Let $\lambda=\lambda_k$ denote the Lebesgue measure on $\I R^k$. The \emph{box} 
(or \emph{grid}, or \emph{upper Minkowski}) \emph{dimension}  of $X\subseteq \I 
R^k$~is
 $$
 \Delta(X):=k-\liminf_{\eps\,\to\, 0^+} \frac{\log \lambda\big(\,\{\V x\in\I R^k\mid \dist(\V x,X)\le \eps\}\,\big)}{\log \eps},
 $$
 where $\dist(\V x, X)$ means e.g.\ the $L^\infty$-distance from the point $\V x$ to the set $X$. 
Let $\partial X$ denote the topological boundary of $X$. It is easy to show that
if $A\subseteq \I R^k$ satisfies $\Delta(\partial A)<k$, then $A$ is Lebesgue
measurable and, furthermore, $\lambda(A)>0$ if and only if $A$ has non-empty 
interior. With these observations, the result of Laczkovich can be formulated as 
follows.

\begin{theorem}[Laczkovich \cite{laczkovich:90,laczkovich:92b,laczkovich:92}]\label{th:L}
Let $k\ge 1$ and let $A, B\subseteq \I R^k$ be bounded sets with non-empty 
interior such that 
$\lambda(A)=\lambda(B)$, $\Delta(\partial A)<k$, and $\Delta(\partial B)<k$. 
Then $A$
and $B$ are equivalent.
\end{theorem}

Theorem~\ref{th:L} applies to circle squaring since the boundary of each of these sets has box dimension~$1$. 
As noted in \cite{laczkovich:92b}, the inequality 
$\Delta(\partial A)<k$ holds if $A\subseteq \I R^k$ is a convex bounded set or if $A\subseteq \I R^2$ has connected boundary
of finite linear measure; thus Theorem~\ref{th:L} applies to such sets as well.

Note that  the
condition that $\lambda(A)=\lambda(B)$ is necessary in Theorem~\ref{th:L}. Indeed, the group 
of translations of $\I R^k$ is amenable (since it is an Abelian group)  and 
therefore the 
Lebesgue measure on $\I R^k$ can be extended to a translation-invariant finitely 
additive measure defined on all subsets (and so equivalent sets which are 
measurable must necessarily have the same measure); see e.g.\ \cite[Chapter 12]{tomkowicz+wagon:btp}
for a detailed discussion. 
Laczkovich~\cite{laczkovich:93} showed that one cannot replace the box dimension 
with the Hausdorff dimension 
in Theorem~\ref{th:L}; see 
also~\cite{laczkovich:03} for further examples of non-equivalent sets.

The proof of Theorem~\ref{th:L} by Laczkovich directly relies on the Axiom of 
Choice in a crucial way. Thus the pieces that he obtains need not be 
measurable. Laczkovich~\cite[Section 10]{laczkovich:90}  writes: \emph{``The 
problem whether or not the circle can be squared with measurable pieces seems to 
be the most interesting.''} 

This problem remained open until now, although some modifications of it were resolved. 
Henle and Wagon (see~\cite[Theorem 9.3]{tomkowicz+wagon:btp}) showed that, for any $\eps>0$, one can square a circle with Borel pieces if one is allowed to use similarities of the plane with scaling factor between $1-\eps$ and $1+\eps$. Pieces can be made even more regular if some larger class of maps can be used (such as arbitrary similarities or affine maps), see e.g.\ \cite{hertel+richter:03,richter:02,richter:05,richter:07}. Also,
if countably many pieces are allowed, then a simple measure exhaustion argument 
shows that, up to a nullset, one can square a circle with measurable pieces (see \cite[Theorem~41]{banach+tarski:24} or~\cite[Theorem~11.26]{tomkowicz+wagon:btp});
the error nullset can be then eliminated by e.g.\ applying Theorem~\ref{th:L}.

The authors of this paper prove in \cite{grabowski+mathe+pikhurko:expansion} that every two bounded measurable
sets $A,B\subseteq \I R^k$, $k\ge 3$, with non-empty interior and of the same measure are equidecomposable
with Lebesgue measurable pieces. In particular, this gives a measurable version 
of Hilbert's third problem (as asked by Wagon~\cite[Question~3.14]{wagon:btp}): 
one can 
split
a regular tetrahedron into finitely many measurable pieces and 
rearrange them into a cube. These 
results
rely on the spectral gap property of  the natural action 
of 
$SO(k)$ on the $(k-1)$-dimensional sphere in $\I R^k$ for $k\ge 3$ and do not 
apply when $k\le 2$. Also, the equidecompositions  obtained in 
\cite{grabowski+mathe+pikhurko:expansion}
cannot be confined to use
translations only. 

Here we fill a part of this gap. Namely,
our main main result (Theorem~\ref{th:Main}) shows that it can be additionally 
required in Theorem~\ref{th:L} 
that all pieces  are Lebesgue measurable.

In fact, Theorem~\ref{th:Main} gives pieces that are also 
\emph{Baire measurable} 
(or \emph{Baire} for short),
that is, each one is the symmetric difference of a Borel set and a meagre 
set. The study of equidecompositions with Baire sets was largely
motivated by \emph{Marczewski's problem} from 1930 whether $\I R^k$ admits a 
non-trivial isometry-invariant
finitely additive  Borel measure that vanishes on bounded meagre 
sets. 
It is not hard to show that the answer is positive for $k\le 2$, see 
e.g.~\cite[Corollary~13.3]{tomkowicz+wagon:btp}. However, the 
problem for $k\ge 3$ remained open for over 60 years until it was resolved
in the negative by Dougherty and Foreman~\cite{dougherty+foreman:92,dougherty+foreman:94} who
proved in particular
that any two bounded Baire 
sets $A,B\subseteq \I R^k$ with non-empty interior are 
equidecomposable with Baire measurable pieces (and thus a ball can be doubled with Baire pieces). A short and elegant 
proof of a
more general
result was recently given by Marks and Unger~\cite{marks+unger:16} 
(see also~\cite{kechris+marks:survey}).
However, as noted in~\cite[Page~406]{marks+unger:16}, the problem whether 
circle
squaring is possible with Baire measurable parts remained open. Also, the 
results 
in~\cite{dougherty+foreman:92,dougherty+foreman:94,marks+unger:16} do not 
apply to the translation equidecomposability $\simTr$, even in 
higher dimensions. Here we 
resolve
these
questions in the affirmative, under the assumptions of Theorem~\ref{th:L}:

\begin{theorem}
\label{th:Main}
Let $k\ge 1$ and let $A, B\subseteq \I R^k$ be bounded sets with non-empty 
interior such that 
$\lambda(A)=\lambda(B)$, $\Delta(\partial A)<k$, and $\Delta(\partial B)<k$.  
Then $A\simTr B$ with parts that are both Baire and 
Lebesgue 
measurable.
\end{theorem}

In addition to implying measurable translation-only versions of Tarski's 
circle squaring, Hil\-bert's 
third problem, and
Walla\-ce-Bolyai-Gerwien's
theorem (a question of Laczkovich~\cite[Page~114]{laczkovich:90}),
 Theorem~\ref{th:Main} also disproves the
following conjecture of Gardner~\cite[Conjecture~5]{gardner:91} for all $k\ge 2$.

\begin{conjecture}\label{cj:gardner} Let $P$ be a polytope and $K$ a convex body in $\I R^k$. If $P$ and $K$ are
equidecomposable with Lebesgue measurable pieces under the isometries  from
an amenable group, then $P$ and $K$ are equidecomposable with convex pieces under the same isometries.
\end{conjecture}

Indeed, for example, let $P$ be a cube and $K$ be a ball of the same volume. It is not hard to show directly that $K$ and $P$ are not equidecomposable with convex pieces, even under the groups of all
isometries of $\I R^k$ for $k\ge 2$. Since Theorem~\ref{th:Main}
uses only translations (that form an  amenable group),  Conjecture~\ref{cj:gardner} is false.

This paper is organised as follows. In \secref{notation} we reduce the problem to the 
torus $\I T^k:=\I R^k/\I Z^k$ 
and state a sufficient condition for measurable equivalence 
in Theorem~\ref{th:Suff}. We also describe there how Theorem~\ref{th:Main} can 
be deduced from Theorem~\ref{th:Suff},
using some results of Laczkovich~\cite{laczkovich:92b}.
The main bulk of this paper consists of the proof of 
Theorem~\ref{th:Suff} in Sections~\ref{proof} and \ref{Baire}. These 
sections are dedicated to respectively Lebesgue and Baire measurability (while
some common definitions and auxiliary results are collected in \secref{common}).
We organised the presentation so that Sections~\ref{proof} and 
\ref{Baire} can essentially be read independently of each other.
\secref{concluding} 
contains some concluding remarks.

In order to avoid ambiguities, a \emph{closed}  (resp.\ \emph{half-open}) 
interval will always mean an interval of \emph{integers} (resp.\ \emph{reals}); 
thus, for 
example,
$[m,n]:=\linebreak\{m,m+1,\dots,n\}\subseteq \I Z$ while $[a,b):=\{x\in \I R\mid a\le 
x<b\}$.
Also, we denote $[n] :=\{1,\dots,n\}$ and $\I N:=\{0, 1, 2, \ldots\}$.

\section{Sufficient condition for measurable equivalence}\label{notation}\label{Suff}

The \emph{$k$-dimensional torus} $\I T^k$
is the quotient of the Abelian group $(\I R^k,+)$ by the subgroup $(\I Z^k,+)$. We
identify $\I T^k$ with the real cube $[0,1)^k$, endowed with  
the addition of vectors~\mbox{modulo}~$1$. 

By scaling the bounded sets $A, B\subseteq \I R^k$ by the same factor and 
translating them,
we can assume that they are subsets of $[0,1)^k$.  
Note that if $A,B\subseteq [0,1)^k$ are 
(measurably) equivalent with translations taken modulo $1$, then they are (measurably) equivalent in $\I R^k$ as well using at most $2^k$ times as many translations. (In fact, if each of $A,B$ has diameter less than $1/2$
with respect to the $L^\infty$-distance, then we do not need to increase the 
number of translations at all.)
So we work inside the torus from now on.

Suppose that we have fixed some vectors $\V x_1,\dots,\V x_d\in \I T^k$ that are \emph{free}, that is, no non-trivial integer
combination of them is the zero element of $(\I T^k,+)$ (or, equivalently, 
$\V x_1,\dots,\V x_d,\V e_1,\dots,\V e_k$, when viewed as vectors in $\I R^k$, 
are linearly independent over the rationals, where $\V e_1,\dots,\V e_k$ are the standard basis vectors of $\I R^k$).

 When reading the following definitions (many of which implicitly depend on $\V x_1,\dots,\V x_d$), the reader is advised to keep in mind
the following connection to Theorems \ref{th:L} and~\ref{th:Main}: we fix some 
large integer $\MD$ and try to establish
the equivalence $A\simTr B$ by translating  only by vectors from the set
 \begin{equation}\label{eq:TranslationSet}
 \VV{\V x_1,\dots,\V x_d}{\MD}:=\left\{\, n_1 \V x_1+\dots+ n_d\V x_d
 \,\mid\, 
\V n\in \I Z^d,\ \|\V n\|_\infty\le \MD\,\right\}.
 \end{equation}
 Thus, if we are successful, 
then
the total number of pieces is at most $|\VV{\V x_1,\dots,\V x_d}{\MD}|= (2\MD+1)^d$.

By a  \emph{coset} of $\V u\in\I T^k$ we will mean the coset taken with respect to the subgroup
of $(\I T^k,+)$ generated by $\V x_1,\dots,\V x_d$, that is, the set $\{\V 
u+\sum_{j=1}^d n_j \V x_j\mid \V n\in \I Z^d\}\subseteq \I T^k$. For  
$X\subseteq \I T^k$, we define
 $$
 X_{\V u}:=\left\{\V n\in\I Z^d \,:\, \V u+n_1 \V x_1+\dots+ n_d\V x_d \in X \right\}.
 $$
  Informally speaking, $X_{\V u}\subseteq \I Z^d$ records which elements of the coset of 
$\V u\in\I 
T^k$ are in $X$. 

If, for every $\V u\in\I T^k$, we have a bijection $\C M_{\V u}:A_{\V u}\to B_{\V u}$ such that
 \beq{eq:PhiU}
 \|\C M_{\V u}(\V n)-\V n\|_\infty\le \MD,\quad\mbox{for
all $\V n\in A_{\V u}$},
 \eeq
 then Theorem~\ref{th:L} follows. Indeed, using the Axiom of Choice select a set $U\subseteq 
\I T^k$ that intersects each coset in precisely one element. Now, each $\V a\in A$
can be uniquely written as $\V u+\sum_{j=1}^d n_j \V x_j$ with $\V u\in U$ and 
$\V n\in\I Z^d$; if we assign 
$\V a$ to the piece which is translated by the vector $\sum_{j=1}^d (m_j- n_j) 
\V x_j$ where $\V m:= \C M_{\V u}(\V n)$, then we get the desired equivalence 
$A\simTr B$. This reduction was used by 
Laczkovich~\cite{laczkovich:90,laczkovich:92b,laczkovich:92}; of course, the 
main challenge  he faced was establishing the existence of the bijections $\C 
M_{\V 
u}$
as in~(\ref{eq:PhiU}). Here, in order to prove
Theorem~\ref{th:Main}, we will additionally need that the family  
$(\C M_{\V u})_{\V u\in\I T^k}$ is consistent for different choices of $\V 
u$ and gives measurable parts.

By an \emph{$n$-cube} $Q\subseteq \I Z^d$ we mean the product of $d$ intervals in $\I Z$ of size $n$,
i.e.\
$Q=\prod_{j=1}^d [n_j,n_j+n-1]$ for some $(n_1,\dots,n_d)\in\I Z^d$. If $n$ is an integer power of $2$,
we will call the cube $Q$ \emph{binary}. 
Given a function $\Phi:\{2^i\mid i\in\I N\}\to \I R$ and a real $\delta\ge 0$, a 
set $X\subseteq\I Z^d$ is called \emph{$\Phi$-uniform (of density $\delta$)} 
if, for every $i\in \I N$ and
$2^i$-cube $Q\subseteq \I Z^d$, we have that
 \beq{eq:PsiUniform}
 \big|\, |X\cap Q|-\delta\, |Q|\,
 \big|\le \Phi(2^i).
 \eeq
 In other words,
this definition says that the discrepancy with respect to binary cubes 
between the counting measure of $X$ and the measure of constant density $\delta$ is upper bounded 
by $\Phi$. A set $Y\subseteq \I T^k$ is called \emph{$\Phi$-uniform (of 
density $\delta$ with respect to $\V x_1,\dots,\V x_d$)} if $Y_{\V u}$ is 
$\Phi$-uniform
of density $\delta$ for every $\V u\in \I T^k$. 

These notions are of interest to us because of the following sufficient 
condition for $A\simTr B$ that  directly follows
from Theorems 1.1 and 1.2 in Laczkovich~\cite{laczkovich:92}.

\begin{theorem}[Laczkovich~\cite{laczkovich:92}] 
\label{th:LSuff}
Let $k,d\ge 1$ be integers, let $\delta>0$, let $\V x_1,\dots,\V 
x_d\in
\I T^k$ be free,  let a function $\Phi:\{2^i\mid i\in\I N\}\to\I R$ satisfy
 \beq{eq:LSumPsiFin}
  \sum_{i=0}^\infty \frac{\Phi(2^i)}{2^{(d-1)i}}<\infty,
  \eeq
 and  let sets $A,B\subseteq \I T^k$
be $\Phi$-uniform of density $\delta$ with
respect to $\V x_1,\dots,\V x_d$. Then $A\simTr B$, using translations that are integer combinations of 
the vectors~$\V x_j$.\end{theorem}

Roughly speaking, the condition (\ref{eq:LSumPsiFin}) states that the discrepancy of $A_{\V u}$ and $B_{\V u}$ with respect to any
$2^i$-cube $Q$ decays noticeably faster than the size of the boundary of $Q$ as $i\to\infty$. 
On the other hand, if
a bijection $\C M_{\V u}$ as in (\ref{eq:PhiU}) exists, then the difference between the number of elements
in $A_{\V u}$ and $B_{\V u}$ that are inside any $n$-cube $Q$ is trivially at most 
$(2\MD+1)^d\cdot 2d\cdot n^{d-1}=O(n^{d-1})$. Theorems~1.1 and 1.5 in~\cite{laczkovich:92}
discuss to which degree the above conditions are best possible.

In this paper we establish the following sufficient condition for measurable 
equivalence.

\begin{theorem}
\label{th:Suff} 
Let $k\ge 1$ and $d\ge 2$ be integers, let $\delta>0$, let $\V x_1,\dots,\V x_{d}\in
\I T^k$ be free,  and let a function $\Psi:\{2^i\mid i\in\I N\}\to\I R$ satisfy
 \beq{eq:SumPsiFin}
  \sum_{i=0}^\infty \frac{\Psi(2^i)}{2^{(d-2)i}}<\infty.
  \eeq
 Define $\Phi:\{2^i\mid i\in\I N\}\to\I R$ by 
$\Phi(2^i):=2^i\cdot\Psi(2^i)$ for $i\in\I N$.
 \begin{enumerate}
  \item\label{it:Lebesgue} If Lebesgue measurable sets $A,B\subseteq \I T^k$ are
$\Psi$-uniform of density $\delta$ with
respect to every $(d-1)$-tuple of distinct vectors from  $\{\V x_1,\dots,\V 
x_{d}\}$, then  $A\simTr B$, where all pieces are Lebesgue measurable
and are translated by integer 
combinations of 
the vectors $\V x_j$.
 \item\label{it:Baire} If Baire sets $A,B\subseteq \I T^k$ are
$\Phi$-uniform of density $\delta$ with
respect to $\V x_1,\dots,\V 
x_{d}$, then  $A\simTr B$, where all pieces are Baire
and are translated by integer 
combinations of 
the vectors $\V x_j$.
 \end{enumerate}
\end{theorem}

\begin{rem}\label{re:uniformities}
In the notation of Theorem~\ref{th:Suff}, if 
$X\subseteq \I T^k$ is $\Psi$-uniform
with respect to any $d-1$ vectors from $\{\V x_1,\dots,\V x_d\}$, then $X$ is 
$\Phi$-uniform with respect to $\V x_1,\dots,\V x_d$. (Indeed, we can 
trivially represent any $d$-dimensional $2^i$-cube
in $\I Z^d$ as the disjoint union of $2^i$ copies of the $(d-1)$-dimensional 
$2^i$-cube.) 
Thus the uniformity assumption of Part~\ref{it:Lebesgue}
is stronger than that of Part~\ref{it:Baire} (or of Theorem~\ref{th:LSuff}). 
We do not know if the $\Phi$-uniformity alone
is sufficient in Part~\ref{it:Lebesgue}.
\end{rem}

The following result of Laczkovich~\cite{laczkovich:92b} shows how to pick vectors that satisfy Theorem~\ref{th:LSuff}. Since it is not explicitly stated in~\cite{laczkovich:92b}, we briefly sketch its proof.

\begin{lemma}[Laczkovich~\cite{laczkovich:92b}]
\label{lm:LX}
Let an integer $k\ge 1$ and a set $X\subseteq \I T^k$ satisfy $\Delta(\partial X)<\nolinebreak k$. Then there is $d(X)$ 
such that, for every $d\ge d(X)$,
if we select uniformly distributed independent random vectors $\V x_1,\dots,\V x_d\in \I T^k$ then
with probability 1 there is $C=C(X;\V x_1,\dots,\V x_d)$ such that $X$ is 
$\Phi$-uniform 
of density $\lambda(X)$ with respect to $\V x_1,\dots,\V x_d$, where 
$\Phi(2^i):=C\cdot 2^{(d-2)i}$ for $i\in\I N$.\end{lemma}

\begin{proof}[Sketch of Proof.] By a \emph{box} in $\I T^k$ we mean a product of $k$ 
sub-intervals of $[0,1)$. Let $d\ge 1$ be arbitrary and let $\V x_1,\dots,\V x_d\in \I T^k$ be random.
By applying the Erd\H os-Tur\'an-Koksma inequality
, 
one can show that, with probability $1$, there is $C'=C'(\V x_1,\dots,\V x_d)$ such that,
for every box $Y\subseteq \I T^k$, $\V u\in\I T^k$, and $N$-cube $Q\subseteq \I Z^d$, we have that
 \beq{eq:ETK}
  \big|\, |Y_{\V u}\cap Q|-\lambda(Y)\, |Q|\,\big| \le \Upsilon(N):=C'\log^{k+d+1} N,
   \eeq
   see \cite[Lemma 2]{laczkovich:92b}.
 In other words, boxes have very small discrepancy with respect to
arbitrary cubes. (In particular, each box is $\Upsilon$-uniform.)

So, assume that (\ref{eq:ETK}) holds and that $\V x_1,\dots,\V x_d$ are free. 
Fix a real $\alpha\in(0,1]$ satisfying $\Delta(\partial X)<k-\alpha$. 
A result of Niederreiter and Wills~\cite[Kollorar~4]{niederreiter+wills:75} implies that the set $X$
is $\Psi$-uniform (with respect to $\V x_1,\dots,\V x_d$) for some $\Psi(N)$ that grows
as $O(\Upsilon(N)^{\alpha /k} N^{d-\alpha d/k})$ as $N\to\infty$.
In particular,
we can satisfy Lemma~\ref{lm:LX} by letting $d(X)$ be any integer such that $\alpha d(X)/k>2$. 
We refer the reader to~\cite[Page~62]{laczkovich:92b} for further details.

Let us also outline the ideas behind~\cite[Kollorar~4]{niederreiter+wills:75} 
in order to
show how the box dimension of $\partial X$ comes into play. 
The definition of $\alpha$ implies that the measure of points within 
$L^\infty$-distance $\eps$ from the boundary of $X$ is at most $\eps^\alpha$ for all small $\eps>0$. 
Let $N$ be large and let $\eps:=\lfloor (N^d/\Upsilon(N))^{1/k}\rfloor^{-1}$. Partition $\I T^k$ into a grid of boxes which is \emph{$\eps$-regular}, meaning 
that side lengths
are all equal to $\eps$.
Let $\C B$ consist of those boxes that intersect $\partial X$. 
By the definition of~$\alpha$,  we have that $|\C B|\le \eps^{\alpha}/\eps^k$.
Next, iteratively merge any two boxes in the interior of
$X$ if they have the same projection on the first $k-1$ coordinates
and share a $(k-1)$-dimensional face. Let $\C I$ be the set of
the final boxes in the interior of $X$. Figure~\ref{fg:circle} illustrates the special case when $X$ is a disk. 
The size of $\C I$ is at most $\e^{-k+1}$ (the number of possible projections)
plus $|\C B|$ (as each box in $\C B$ can ``prevent'' at most one merging). 

\begin{figure}[t]
\begin{center}
\includegraphics[height=3.7cm]{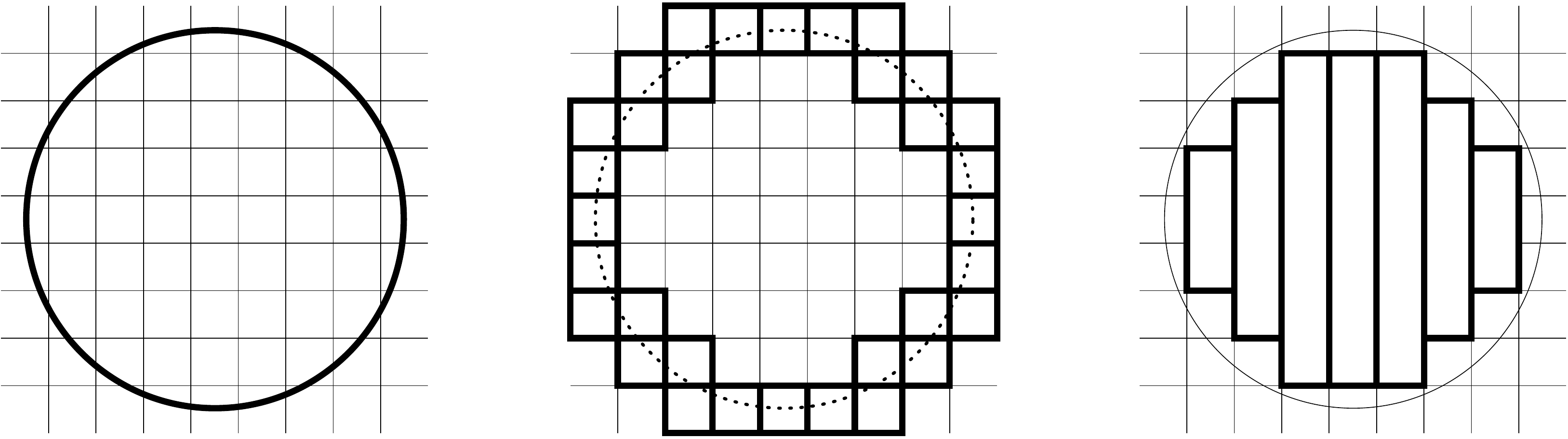}
\end{center}
\caption{\quad i) Circle $\partial X$ and $\eps$-regular grid;\quad ii) boxes in $\C B$;\quad iii) boxes in $\C I$}\label{fg:circle}
\end{figure}

Pick any $\V u\in\I T^k$ and an $N$-cube 
$Q\subseteq \I Z^d$. 
We take the dual point of view where
we fix $Q':=\{\V u+\sum_{j=1}^d n_j \V x_j\mid \V n\in Q\}\subseteq \I T^k$
and measure its discrepancy with respect to boxes. Namely, we have by~(\ref{eq:ETK}) that 
$D'(Y)\le \Upsilon(N)$  for every box $Y\subseteq \I T^k$, where we define $D'(Y):=\big|\,|Q'\cap Y|-\lambda(Y)\, N^d\,\big|$. 
This implies that 
 $$
  D'(X)\le \sum_{Y\in\C I} D'(Y)+\sum_{Y\in\C B} D'(Y\cap X) \le |\C I|\cdot\Upsilon(N)+|\C B|\, (\eps^kN^d+\Upsilon(N)),
   $$
 giving the stated upper bound after routine simplifications.\end{proof}

Thus Lemma~\ref{lm:LX} shows that the uniformity assumption of Part~\ref{it:Baire} of Theorem~\ref{th:Suff}
can be satisfied if the sets $A$ are $B$ are as in Theorem~\ref{th:Main}. The lemma also suffices for Part~\ref{it:Lebesgue} of 
Theorem~\ref{th:Suff}, thus leading to the proof of Theorem~\ref{th:Main} 
 as follows.

\begin{proof}[Proof of Theorem~\ref{th:Main}.] Observe that the assumption 
$\Delta(\partial X)<k$ implies that $X\subseteq \I T^k$ is both Baire and 
Lebesgue
measurable. For example, let us argue that $X$ is Baire. Every set is the union
of its interior (an open set) and a subset of its boundary. So it 
is enough
to show that $\partial X$ is nowhere dense. Take any ball $U\subseteq\I T^k$ of 
radius $r>0$. As $\e\to0$, the $\e$-neighbourhood of $\partial X$ has 
measure at most $\e^{\alpha}$ for some constant $\alpha>0$. This is 
strictly smaller
than $((r-\e)/r)^k\,\lambda(U)$, the volume of  the ball $U'$ concentric to $U$ of
radius $r-\e$, so at least one point $\V x\in U'$ is uncovered. The open ball 
of radius $\e$ around $\V x$ lies entirely inside $U$ and avoids $\partial X$. 
Thus $\partial X$ is 
nowhere dense, as desired.

Therefore, the sets $A$ and $B$ in Theorem~\ref{th:Main} are both Baire and 
Lebesgue measurable. Next, let us show that we can satisfy the uniformity 
assumption of Part~\ref{it:Lebesgue} of Theorem~\ref{th:Suff}.\

Let 
$d:=\max(d(A),d(B))+1$, where $d(X)$ is the function 
provided by Lemma~\ref{lm:LX}. Fix  free vectors $\V x_1,\dots,\V x_d\in \I T^k$ such that
every $(d-1)$-tuple of them satisfies the conclusion of Lemma~\ref{lm:LX} for both $A$ and $B$. Such vectors
exist since the desired properties hold with probability $1$ if we sample 
the vectors $\V x_j$  independently.
Let  $C<\infty$ be the maximum,
over all choices of $X\in \{A,B\}$ and integers
$1\le i_1<\dots<i_{d-1}\le d$,
of the corresponding constants 
$C(X;\V x_{i_1},\dots,\V x_{i_{d-1}})$. Then the assumptions of 
Part~\ref{it:Lebesgue} of
Theorem~\ref{th:Suff} hold with 
$\Psi(2^i):=C\cdot 2^{(d-3)i}$ (and the same function works with
Part~\ref{it:Baire}).

Theorem~\ref{th:Suff} implies that $A$ and $B$ are
equivalent with Baire (resp.\ Lebes\-gue) measurable pieces. Allowing empty 
pieces, let this be witnessed respectively by 
partitions $A=\cup_{\V v\in \C V} A_{\V v}'$
and $A=\cup_{\V v\in \C V} A_{\V v}''$ for some finite 
$\C V\subseteq \I T^k$, where the pieces $A_{\V v}'$ and $A_{\V v}''$ are 
translated
by~$\V v$. 
These equidecompositions
can be ``merged'' as follows. Take a nullset $X\subseteq \I T^k$ such that
$\I T^k\setminus X$ is meagre; the existence of  $X$ follows from
e.g.~\cite[Theorem~1.6]{oxtoby:mc}. We can 
additionally assume that $X$ is invariant
under all translations from~$\C V$. (For example, 
take
the union of all translates of $X$ by integer combinations of the vectors
from $\C V$; it is still a nullset since we take countably many translates.)
Now, we combine the  
Baire partition of $A$ restricted to $X$ with the Lebesgue partition restricted to 
$\I T^k\setminus X$. Specifically, let $A_{\V v}:=(A_{\V v}'\cap X)\cup 
(A_{\V v}''\setminus X)$ for $\V v\in \C V$. Clearly, these 
sets partition $A$ while, by the invariance of $X$, the corresponding 
translates $A_{\V v}+\V v$, for $\V v\in 
\C V$, partition $B$. Also, each part $A_{\V v}$ is both 
Baire and Lebesgue measurable. This proves 
Theorem~\ref{th:Main}.\end{proof}

\section{Some common definitions and results}\label{common}

Our proofs of Parts 1 and 2 of Theorem~\ref{th:Suff} proceed 
somewhat differently. This section collects some definitions and auxiliary 
results that are common to both parts.
Here, let \emph{measurable} mean Baire or Lebesgue measurable, depending on
which $\sigma$-algebra we are interested in.

Since we will study equidecompositions from graph-theoretic point of view, we 
find it convenient to adopt some notions of graph theory to our purposes 
as follows. 

By a \emph{bipartite graph} we mean a triple $G=(V_1,V_2,E)$, where  $V_1$ and 
$V_2$
are (finite or infinite) \emph{vertex sets} and $E\subseteq V_1\times V_2$ is a 
set of \emph{edges}. (Note that
$E$ consists of ordered pairs to avoid ambiguities when $V_1$ and $V_2$ are not 
disjoint.) The subgraph \emph{induced} by sets $X_1$ and $X_2$ is
 \begin{equation}\label{eq:induced}
 G[X_1,X_2]:=\big(\,V_1\cap X_1,\,V_2\cap X_2,\,E\cap (X_1\times X_2)\,\big).
 \end{equation}
A \emph{matching} in $G$ is a subset $\C M$ of $E$ which gives a partial 
injection from
$V_1$ to $V_2$ (that is, if $(a,b)$ and $(a',b')$ are distinct pairs in
$\C M$ then $a\not=a'$ and $b\not=b'$). In fact, we will identify a matching 
with
the corresponding partial injection. In particular, the
sets of matched points in $V_1$ and $V_2$ can be respectively denoted by $\C 
M^{-1}(V_2)$
and $\C M(V_1)$. A matching $\C M$ is \emph{perfect} if it is a bijection
from $V_1$ to $V_2$ (that is, if $\C M(V_1)=V_2$ and $\C M^{-1}(V_2)=V_1$).  
For a set $X$ lying in one part of $G$, let its \emph{neighbourhood} $\Gamma(X)$
consist of those vertices in the other part that are connected by 
at least one edge to $X$. (In the functional notation, we have 
$\Gamma(X)=E(X)$ for 
$X\subseteq A$ and
$\Gamma(X)=E^{-1}(X)$ for $X\subseteq B$.) If $G$ is \emph{locally finite}
(that is, every \emph{degree} $|\Gamma(\{x\})|$ is finite), then Rado's 
theorem~\cite{rado:49} states that $G$ has a perfect matching if and only if
 \begin{equation}\label{eq:Rado}
 |\Gamma(X)|\ge |X|,\quad\mbox{for every finite subset $X$ of $A$ or 
$B$}.  
 \end{equation}
 Note that 
if $V_1\cap V_2=\emptyset$ then we get the standard notions of graph theory  
with respect to the corresponding 
undirected graph on $V_1\cup V_2$.

Thus, an equidecomposition between $A,B\subseteq \I T^k$ where all translations
are restricted to the set $\VV{\V x_1,\dots,\V 
x_d}{\MD}$ that was defined in~\eqref{eq:TranslationSet} is nothing else than a 
perfect matching in the bipartite graph
 \begin{equation}\label{eq:CG}
  \C G:=(A,B,E),
 \end{equation}
 where $E$ consists of all pairs $(\V a,\V b)\in A\times B$ with 
 $\V 
b-\V a\in \VV{\V x_1,\dots,\V x_d}{\MD}$. 

Assume from now on that both $A$ and $B$ are measurable (which will be the case in
all applications). Then, each of the vertex parts of the graph $\C G$
is additionally endowed with the $\sigma$-algebra of measurable sets;
objects of this type appear in orbit equivalence~\cite{kechris+miller:toe},
limits of sparse graphs~\cite{lovasz:lngl}, and other areas.
A matching $\C M$ in $\C G$ is called 
\emph{measurable} if the set $\{\V a\in 
A\mid \C M(\V a)-\V 
a=\V v\}$ is measurable for each $\V v\in\VV{\V x_1,\dots,\V x_d}{\MD}$.

Also, we will consider the subgraphs of $\C G$ induced 
by cosets, viewing these as graphs on subsets of $\I Z^d$. 
Namely, for $\V u\in \I T^k$, consider the bipartite graph $\C G_{\V u}:=(A_{\V 
u},B_{\V u},E_{\V u})$, where
 $$
 E_{\V u}:=\big\{\, (\V a,\V b)\in A_{\V u}\times B_{\V u} \mid 
 \|\V a-\V b\|_\infty
 \le \MD\,\big\}.
 $$
 Again, a bijection $\C M_{\V u}:A_{\V u}\to B_{\V u}$ as in (\ref{eq:PhiU}) is 
nothing
else than a perfect matching in $\C G_{\V u}$ and, in order to prove
Theorem~\ref{th:LSuff}, it is enough to show that each $\C G_{\V u}$ has at 
least one perfect matching. 
For the proof of Theorem~\ref{th:Suff}, we will also need that the dependence on 
$\V u$ is ``equivariant''
and  ``measurable'' in the following sense.

Namely, we call the family $(\C M_{\V u})_{\V u\in\I T^k}$ with $\C M_{\V u}$ 
being a  matching in $\C G_{\V u}$ \emph{equivariant} if, for all $\V u\in \I 
T^k$ and $\V n\in\I Z^d$, we have
 \beq{eq:Equivariance}
 \C M_{\V u+n_1\V x_1+\dots+n_d\V x_d}=
 \{(\V a-\V n,\V b-\V n)\mid (\V a,\V b)\in \C M_{\V u}\}.
 \eeq

Note that, if  (\ref{eq:Equivariance}) holds, then we can define a partial 
injection $\C M:A\to B$ as follows.
In order to
find the image $\C M(\V a)$ of $\V a\in A$, take any $\V u$ such that $\V a$ is 
in the coset of $\V u$,
say $\V a=\V u+\sum_{j=1}^d n_j \V x_j$ with $\V n\in\I Z^d$. Note that $\V 
n\in 
A_{\V u}$. If $\V n$ is not matched by $\C M_{\V u}$, then let $\C M(\V a)$ be 
undefined; otherwise let $\C M(\V a):=\V u+\sum_{j=1}^d m_j \V x_j\in B$, where 
$\V m:=\C M_{\V u}(\V n)$.
 It is easy to see that, by (\ref{eq:Equivariance}), the definition of $\C M(\V 
a)$ does not depend on the choice of~$\V u$ (and it will often be 
notationally convenient to take $\V u=\V a$).

The concept of equivariance can be applied to other kinds of objects, with the 
definition 
being the obvious adaptation of (\ref{eq:Equivariance}) in all cases that we 
will encounter. Namely,
the ``meta-definition'' is that if we shift the coset reference point from $\V 
u$ to $\V u+n_1\V x_1+\dots+n_d\V x_d$ 
for some $\V n\in \I Z^d$, then the object does not change, i.e.\ its new 
coordinates are
all shifted by $-\V n$. 
For example, for every $X\subseteq
\I T^k$ the
family of sets $(X_{\V u})_{\V u\in\I T^k}$ is equivariant and, conversely, 
every equivariant family of
subsets of $\I Z^d$ gives a subset of $\I T^k$. As another example, the family 
$(\C G_{\V u})_{\V u\in\I T^k}$
is equivariant and corresponds to the bipartite graph $\C G$ defined 
in~\eqref{eq:CG}.

We call an equivariant family $(\C M_{\V u})_{\V u\in\I T^k}$ with $\C M_{\V u}$ 
being a 
(not necessarily perfect) matching in $\C G_{\V u}$ \emph{measurable} if the
natural encoding of the corresponding matching $\C M$ by a function $\I T^k\to [-\MD,\MD]^d\cup\{\mbox{\sc 
unmatched}\}$ is measurable. (Note that this is equivalent to the 
measurability of $\C M$ as defined 
after~\eqref{eq:CG}.)
\comment{(For example, to see the converse implication
note that, for $\V n\in [-\MD,\MD]^d$, the set $g^{-1}(\V n)$ consists of those 
$\V a\in \C M^{-1}(B)$ for which
$\C M(\V a)=\V a+\sum_{j=1}^d n_j\V x_j$.)}%
  Again, this concept can be applied to other objects: 
for example, an equivariant family $(X_{\V u})_{\V u\in\I T^k}$ of subsets of 
$\I Z^d$ is called \emph{measurable} if the corresponding encoding $\I 
T^k\to\{0,1\}$ (i.e. the corresponding set $X\subseteq \I T^k$)
is measurable.

Thus, if we can find an equivariant and measurable family $(\C M_{\V u})_{\V 
u\in\I T^k}$ with
$\C M_{\V u}$ being a perfect matching
in $\C G_{\V u}$ for each $\V u\in\I T^k$, then we have a measurable bijection 
$\C M:A\to B$. Furthermore, the differences
$\C M(\V a)-\V a$ for $\V a\in A$ are all restricted to the finite set 
$\VV{\V x_1,\dots,\V x_d}{\MD}$, giving the required measurable equivalence 
$A\simTr B$.

Thus, informally speaking, each element $\V n\in A_{\V u}$ has to find its 
match 
$\C M_{\V u}(\V n)$ in a measurable way which is also invariant under shifting 
the whole coset by any integer vector.  For 
example, 
a measurable inclusion-maximal matching $\C M$ between $A$ and $B$ can be 
constructed by 
iteratively applying
the following over all $\V v\in \VV{\V x_1,\dots,\V x_d}{\MD}$: 
add to the current matching $\C M$ all possible pairs  $(\V a,\V a+\V v)$, 
i.e.\ for all $\V a$ in the set 
 $$
 X:= \big(A\setminus \C M^{-1}(B)\big)\cap \big((B\setminus\C M(A))-\V v\big).
 $$ 
 Clearly, $X$ is measurable if $\C M$ is; thus one iteration preserves
the measurability of $\C M$. Also, each of the 
above iterations can be determined by a ``local'' rule 
within a coset: namely, the 
match of a vertex $\V n\in A_{\V u}$ depends only on the current picture 
inside the ball of radius $\MD$ around $\V n$ (while the new
values of $\C M$ can be determined in parallel).

Let us formalise the above idea. For $r\in\I N$, a \emph{radius-$r$ local rule} 
(or simply an \emph{$r$-local rule}) is a function
$\C R:\I N^{Q_r}\to \I N$, where $Q_r:=[-r,r]^d\subseteq \I Z^d$; it instructs 
how to transform any
function $g:\I T^k\to\I N$ into another function $g^{\C R}:\I T^k\to\I N$. 
Namely, for $\V u\in\I T^k$, 
we define 
$$
 g^{\C R}(\V u):=\C R(g_{\V u}|_{Q_r}),
  $$
   where $g_{\V u}:\I Z^d\to\I N$ is the coset version of $g$ (i.e.\
$g_{\V u}(\V n):=g(\V u+\sum_{j=1}^d n_j\V x_j)$ for
$\V n\in\I Z^d$)  and $g_{\V u}|_{Q_r}:Q_r\to \I N$ denotes its restriction
to  the cube $Q_r$.

\begin{lemma}\label{lm:local}  If $g:\I T^k\to\I N$ is a measurable function, 
then, for 
any $r$-local rule $\C R$, 
the function $g^{\C R}:\I T^k\to\I N$ is measurable.\end{lemma}

 \begin{proof}
 For any function
$f:Q_r\to \I N$ define  
 $$
 X_f:=\{\V u\in\I T^k\mid g_{\V u}|_{Q_r}=f\}. 
 $$
 Thus $\V u\in X_{f}$ if and only if $g(\V u+\sum_{j=1}^d n_j\V x_j)=f(\V n)$
for every $\V n\in Q_r$. This means that $X_{f}$ is the intersection,  over $\V 
n\in Q_r$, of the translates of $g^{-1}(f(\V n))\subseteq\I T^k$ by the vector 
$-\!\sum_{j=1}^d n_j\V x_j$. Each of these translates is a measurable set
by the measurability of $g:\I T^k\to\I N$. 

Furthermore, the pre-image of any $i\in\I N$ 
under $g^{\C R}$ is the disjoint union of $X_f$ over $f$ with $\C R(f)=i$. 
This union is measurable as there are only countably many possible 
functions~$f$.\end{proof}

To
avoid confusion when we have different graphs on $\I Z^d$, the  
\emph{distance} between $\V x,\V y\in \I Z^d$ will always mean the 
$L^\infty$-distance between vectors: 
 $$
  \dist(\V x,\V y)=\|\V x-\V 
y\|_\infty.
 $$
 Also, we use the standard definition of the distance
between sets: 
  \begin{equation}\label{eq:SetDistance}
  \dist(X,Y):=\min\{\dist(\V x,\V y)\mid 
\V x\in X,\ \V y\in Y\},\quad X,Y\subseteq \I Z^d.
 \end{equation}
 For $X\subseteq \I 
Z^k$ and $m\in \I N$, we define the \emph{$m$-ball around $X$} to be
 $$
 \dist_{\le m}(X):=\{\V n\in\I Z^d\mid \dist(\V n,X)\le m\}.
 $$
A collection $\C X$ of elements or subsets of $\I Z^d$ is \emph{$r$-sparse} 
if the distance between any two distinct members of $\C X$ is strictly larger
than~$r$. A set $X\subseteq \I T^k$ is called \emph{$r$-sparse} if 
$X_{\V u}\subseteq \I Z^d$ is $r$-sparse for each $\V u\in\I T^k$.

\begin{lemma}\label{lm:KST} For every $r$ there is a Borel measurable
map $\chi:\I T^k\to [t]$ for some $t\in\I N$ such that each pre-image 
$\chi^{-1}(i)\subseteq \I T^k$, $i\in [t]$, is $r$-sparse.
\end{lemma}
 \begin{proof} The existence of $\chi$ follows from the more general 
results of Kechris, Solecki and  
Todorcevic~\cite{kechris+solecki+todorcevic:99}. 

Alternatively, 
pick $n\in\I N$ such that $1/n$ is smaller than the minimum distance 
inside the finite 
set 
$\{\,\sum_{j=1}^d n_j\V x_j\mid \V n\in\nolinebreak\I Z^d,\ \|\V n\|_\infty\le\nolinebreak r\,\}
\subseteq\I T^k$. Then any subset of $\I T^k$ of 
diameter at most $1/n$ is $r$-sparse. Thus we can take for $\chi$ 
any function that has the 
half-open boxes of the $(1/n)$-regular grid on $\I T^k$ as its pre-images (where
$t=n^k$).\end{proof}

For 
$X\subseteq \I Z^d$,
its \emph{boundary} $\partial X$ is the set of ordered pairs $(\V m,\V n)$ such 
that $\V m\in X$, $\V n\in \I Z^d\setminus X$ and the vector $\V n-\V m$ has 
zero entries except one entry equal to $\pm1$ (i.e.\
$\V n-\V m=\pm\V e_j$ for a standard basis vector $\V e_j$). The 
\emph{perimeter} of $X$ is $p(X):=|\partial X|$. In other words, the perimeter 
of $X$ is the number of edges leaving $X$ in 
the standard $2d$-regular graph on $\I Z^d$. 

We will need a lower bound on the perimeter of a finite set 
$X\subseteq \I Z^d$ in terms of its size. While
the exact solution to this edge-isoperimetric problem is known (see Ahlswede 
and 
Bezrukov~\cite[Theorem~2]{ahlswede+bezrukov:95}), we find it more convenient to
use the old result of Loomis and Whitney~\cite{loomis+whitney:49} that gives a 
bound which is easy to state and suffices for our purposes.

\begin{lemma}\label{lm:Isop} For every finite $X\subseteq \I Z^d$ we have 
$p(X)\ge 2d\cdot |X|^{(d-1)/d}$.
\end{lemma}
\begin{proof}  A result of Loomis and Whitney~\cite[Theorem~2]{loomis+whitney:49} directly 
implies that $|X|^{d-1}\le \prod_{j=1}^d |X_j|$, where $X_1,\dots,X_d\subseteq \I Z^{d-1}$ are all
$(d-1)$-dimensional projections of $X$. Thus, by the Geometric--Arithmetic Mean Inequality, we
obtain the required:
 $$
  p(X)\ge 2\sum_{j=1}^d |X_j| \ge 2d \left(\prod_{j=1}^d |X_j|\right)^{1/d}\ge 2d\cdot |X|^{(d-1)/d}.
   $$
   \end{proof}

\section{Proof of Part~\ref{it:Lebesgue} of
Theorem~\ref{th:Suff}}\label{proof}

Throughout this section, \emph{measurable} means Lebesgue measurable.

\subsection{Overview of main ideas and steps}\label{outline}

First, let us define some global constants that will be used for proving 
Part~\ref{it:Lebesgue} of Theorem~\ref{th:Suff}.
Recall that we are given  the measurable sets $A,B\subseteq \I T^k$ that are 
$\Psi$-uniform of density $\delta>0$ with respect to any $d-1$ of the vectors $\V 
x_1,\dots,\V x_d\in\I T^k$.  As we mentioned in Remark~\ref{re:uniformities},
this implies that 
$A$ and $B$ are $\Phi$-uniform with respect to $\V x_1,\dots,\V x_d\in\I T^k$.
(Recall that $\Phi(2^i):=2^i\cdot \Psi(2^i)$ for $i\in\I N$.) It easily follows
(e.g.\ from Lemma~\ref{lm:density} below) that $\lambda(A)=\lambda(B)=\delta$.

Given $A,B,\Psi,\V x_1,\dots,\V x_d$, choose a large constant $\MD$ (namely, it has to satisfy
Lemma~\ref{lm:ShortAug} below). Let $(N_i)_{i\in\I N}$ be a strictly 
increasing sequence, consisting of integer powers of $2$ 
such that $\sum_{i=0}^\infty N_i^2/N_{i+1}<\infty$.  
When some index $i$ goes to infinity, we may use asymptotic notation, such as $O(1)$, to denote 
constants that do not depend on $i$.

We will be constructing the desired measurable perfect matching in the 
bipartite graph $\C G=(A,B,E)$ that was defined by~\eqref{eq:CG} 
by iteratively improving
partial matchings. 
Namely, each Iteration~$i$ replaces the previous partial measurable 
matching $\C M_{i-1}$ by a ``better'' matching $\C M_{i}$ using finitely many 
local rules. Clearly, the new family $(\C M_{i,\V u})_{\V u\in\I T^k}$ is still 
equivariant and, by Lemma~\ref{lm:local}, measurable. We wish to find matchings 
$(\C M_i)_{i\in\I N}$ such  that for a.e.\ (almost every)
$\V a\in A$ the sequence $\C M_i(\V a)$ stabilises eventually, that is, there are $n\in\I N$ and $\V b\in B$ such that
$\C M_i(\V a)=\V b$ for all $i\ge n$. In this case, we agree that the final partial map $\C M$
maps $\V a$ to $\V b$. Equivalently, 
 \beq{eq:CM}
  \C M:=\cup_{i\in \I N} \cap_{j=i}^\infty \C M_{j},
 \eeq 
  where we view matchings in $\C G=(A,B,E)$ as subsets of $E$.
 Clearly, any family $(\C M_{\V u})_{\V u\in \I T^k}$ of matchings obtained this way is equivariant and measurable.

In order to guarantee that almost every vertex of $A$ is matched (i.e.\ that
$\lambda(\C M^{-1}(B)\setminus\nolinebreak A)=0$), it is enough to establish the following two properties:
 \begin{align}
 \lim_{i\to\infty} \lambda\!\left(\C M_{i}^{-1}(B)\right) &= \lambda(A),\label{eq:aim1}\\
 \sum_{i=0}^\infty \lambda\!\left((\C M_i\triangle \C M_{i+1})^{-1}(B)\right) & < \infty,\label{eq:aim2}
 \end{align}
 where 
$\C M_i\triangle \C M_{i+1}\subseteq E$ is the 
symmetric difference of $\C M_i$ and $\C M_{i+1}$, and thus $(\C M_i\triangle \C M_{i+1})^{-1}(B)$
is the set of those $\V a\in A$ such that $\C M_i(\V a)\not=\C M_{i+1}(\V a)$, including the cases when only
one of these is defined. 

Indeed, suppose that (\ref{eq:aim1}) and (\ref{eq:aim2}) hold. Let $A_i'$
 consist
of those vertices of $A$ whose match is modified at least once after Iteration~$i$, that is,   $A'_{i}:=\cup_{j=i}^\infty (\C M_j\triangle \C M_{j+1})^{-1}(B)$. The measure 
$\lambda(A'_{i})$ tends to~$0$ as $i\to \infty$ because it is trivially bounded by the corresponding
tail of the sum in~(\ref{eq:aim2}). Thus the set $A':=\cap_{i=0}^\infty A'_{i}$ of
vertices in $A$ that do not stabilise eventually has measure zero. Also, for 
every $i\in\I N$ we have that $\C M^{-1}(B)\supseteq \C M_i^{-1}(B)\setminus 
A'_{i}$. 
If we consider the measure of these sets and use~(\ref{eq:aim1}),
we conclude that $\lambda(\C M^{-1}(B))\ge \lambda(A)$, giving the required conclusion.

Thus, if we are successful in establishing (\ref{eq:aim1}) and (\ref{eq:aim2}), 
this gives an a.e.\ defined measurable map $\C M$, which shows that $A\setminus 
A'$ and $B\setminus B'$
are measurably equivalent, for some nullsets $A'\subseteq A$ and $B'\subseteq B$. 
It is not hard to modify $\C M$ to get rid of the exceptional sets. Namely, let 
$X\subseteq \I T^k$ be the union of all cosets that intersect $A'\cup B'$.
Note that $X$ is a 
nullset. 
Let $\C M':A\to B$ be given by Theorem~\ref{th:LSuff} using the same vectors $\V x_1,\dots,\V x_d$. Then, $\C M$ (resp.\ $\C M'$) induces a bijection $A\setminus X\to B\setminus X$ (resp.\ 
$A\cap X\to B\cap X$) 
and we can use $\C M$ on $A\setminus X$ and $\C M'$ on $A\cap X$. The obtained 
bijection $A\to B$ is measurable since $\C M'$ is applied only inside the 
nullset~$X$.

The following trivial observation will be enough in all our forthcoming estimates of the measure of ``bad''
sets. We say that a set $X\subseteq \I T^k$ (or an equivariant family $(X_{\V u})_{\V u\in\I T^k}$) 
\emph{has  uniform density at most} $c$ if there is $r\in \I N$ such that for every $\V u\in\I T^k$ and for every $r$-cube $Q\subseteq \I Z^d$ we have
$|X_{\V u}\cap Q|\le cr^d$. 

\begin{lemma}\label{lm:density} If a measurable set $X\subseteq \I T^k$ has uniform density at most $c$,
then $\lambda(X)\le c$.\end{lemma}

\begin{proof} Let $r\in\I N$ witness the stated uniform density.
Consider $r^d$ translates $X+\sum_{j=1}^d n_j\V x_j$ over $\V n\in [r]^d$. By our 
assumption, every point of $\I T^k$ is covered at most $cr^d$ times. Thus the lemma follows from the
finite additivity and translation invariance of the Lebesgue measure~$\lambda$.\end{proof}

Since our construction of the matching $\C M_i$ involves ``improving''
$\C M_{i-1}$ along 
augmenting paths, let us give the corresponding general definitions now. 
Given a 
matching 
$\C M$ in a bipartite graph $G=(V_1,V_2,E)$, an \emph{augmenting path} is a 
sequence $P=(v_0,\dots,v_m)$ of vertices
such that $v_0\in V_1\setminus \C M^{-1}(V_2)$, $v_m\in V_2\setminus \C M(V_1)$,
$(v_{i},v_{i-1})\in \C M$ for 
all even $i\in [m]$, and $(v_{i-1},v_i)\in E\setminus\C M$ for all odd $i\in 
[m]$. In other words,
we start with an unmatched vertex of $V_1$ and alternate between edges in 
$E\setminus\C M$
and $\C M$ until we reach an unmatched vertex of $V_2$; note that all even 
(resp.\ odd) numbered vertices necessarily belong to the same part and are 
distinct. The \emph{length} of $P$ is $m$,
the number of edges in it;
clearly, it has to be odd. If we \emph{flip} the path $P$, that is, remove 
$(v_{i},v_{i-1})$ from $\C M$
for all even $i\in [m]$ and add $(v_{i-1},v_{i})$ to $\C M$ for all odd 
$i\in[m]$,
then we obtain another matching that improves $\C M$ by covering two extra 
vertices. A matching in a finite graph  is \emph{maximum} 
if it has the largest number of edges among all matchings.

As we already mentioned, we try to achieve (\ref{eq:aim1}) and (\ref{eq:aim2}) 
by iteratively flipping augmenting paths 
using some local rules.  
We have to be careful how we guide the paths since it is not a priori clear that 
if two  unmatched points from different parts are close to each other in  $\C G_{\V u}$, then there 
is a relatively short augmenting path (or any augmenting path at all). 

The following lemma gives us some control over this. A \emph{rectangle} $R\subseteq \I Z^d$ is the product
of $d$ finite intervals of integers, $R=\prod_{j=1}^d[a_j,b_j]$. 
Its \emph{side lengths} are $b_j-a_j+1$, $j\in [d]$. We say that $R$ is 
\emph{$\rho$-balanced} if the ratio of any two side lengths is at most $\rho$.

\begin{lemma}\label{lm:ShortAug} Let the assumptions of 
Part~\ref{it:Lebesgue} of Theorem~\ref{th:Suff} hold and let 
$\MD=\MD(A,B,\Psi,\V x_1,\dots,\V x_d)$ be sufficiently large. Take arbitrary 
$\V u\in\I T^k$ and a $3$-balanced rectangle $R\subseteq \I Z^d$.
If $\C M$ is a  matching in 
$\C G_{\V u}[R,R]$ (the subgraph of $\C G_{\V u}$ induced by $R$, as defined
in~\eqref{eq:induced}) 
that 
misses at least one vertex in each part, then $\C G_{\V u}[R,R]$ contains an 
augmenting 
path whose length is at most the maximum side length of $R$.
In particular, every maximum matching in $\C G_{\V u}[R,R]$ completely covers 
one part of the graph.\end{lemma}

Surprisingly, this combinatorial lemma (which, as we will see later, 
relies only on the $d$-dimensional $\Phi$-uniformity of $A$ and $B$) is quite 
difficult to prove. Although much of work needed for its
proof was already done 
by Laczkovich~\cite{laczkovich:92}, a rather long argument is still required to complete it, so
we postpone all details to \secref{ShortAug}.

Given Lemma~\ref{lm:ShortAug}, another idea that went into the proof is the following. Given a partition of $(\I T^k)_{\V u}\cong \I Z^d$ into a regular grid of
$2^j$-cubes with a maximum matching inside each cube, group the cubes $2^d$ apiece so that the
new groups form a $2^{j+1}$-regular grid. By Lemma~\ref{lm:ShortAug},
the number of unmatched 
vertices inside each $2^j$-cube $Q$ is at most $|\, |A_{\V u}\cap Q|-|B_{\V u}\cap Q|\,|$,
which is at most $2\Phi(2^j)$ by the assumptions of Theorem~\ref{th:Suff}. In particular, 
the uniform density of unmatched points tends to $0$ with $j\to\infty$, helping with (\ref{eq:aim1}).
Inside each new $2^{j+1}$-cube
$Q'$,
iteratively select and flip an augmenting path of length at most $2^{j+1}$ until none exists. 
By Lemma~\ref{lm:ShortAug}, we have a maximum matching inside $Q'$ at the end. The total number of changed 
edges is at most $2^{j+1}\cdot 2^d\cdot 2\Phi(2^j)$. 
If we iterate over all $j\in \I N$ and sum the density of these changes, we get 
 \beq{eq:Estimate}
  \sum_{j=0}^\infty \frac{2^{j+1}\cdot 2^d\cdot 2\Phi(2^j)}{(2^{j+1})^d} =4\, \sum_{j=0}^\infty 
\frac{\Phi(2^j)}{2^{(d-1)j}}.
 \eeq 
 The above sum converges by (\ref{eq:SumPsiFin}), giving  a ``coset analogue'' of 
the desired requirement (\ref{eq:aim2}).

However, it is impossible to construct a perfect partition of each
coset into cubes of the same side length $N\ge 2$ in an equivariant and measurable way (because
the $\I Z^d$-action on $\I T^k$ given by the translations by $N\V x_1,\dots,N\V x_d$ is ergodic for typical vectors $\V x_j$). We overcome this issue 
by fixing, at each Iteration $i$, some set $S_i\subseteq \I T^k$ such that the elements of $S_{i,\V u}\subseteq \I Z^d$ (called \emph{seeds}) are far apart from each other. 
Informally speaking, we view each seed $\V s\in S_{i,\V u}$ as a processor that ``controls'' its Voronoi cell;
namely, $\V s$ draws the regular grid
$\C Q_i$ consisting of $N_{i}$-cubes inside its
Voronoi cell, treating itself as the centre of the coordinate system. We obtain what looks
as an $N_i$-regular grid except possible misalignments near cell boundaries. Also, assume that ``most'' of $\I Z^d$ is already covered
by grid-like areas of $N_{i-1}$-cubes with each cube containing a maximum matching that were constructed in the previous iteration step. Now,
each $\V s\in S_{i,\V u}$ aligns these as close as possible to its $N_{i}$-grid and then uses 
Lemma~\ref{lm:ShortAug} to do incremental steps as in the previous paragraph, running them from
$j=\log_2N_{i-1}$ to $\log_2N_{i}-1$ until every $N_i$-cube that is under control of $\V s$ 
induces a maximum matching. Of course, the possible 
misalignments of the grids and boundary issues require extra technical arguments.  (This is the part where we need
the $(d-1)$-dimensional $\Psi$-uniformity.)

We hope that the above discussion will be a good guide for understanding the proof
of Part~\ref{it:Lebesgue} of Theorem~\ref{th:Suff} which we present now.

\subsection{Details of the proof}\label{ProofReally}

We will use the global constants that were defined at the beginning of \secref{outline}.

\subsubsection{Constructing the seed set $S_i$}\label{Si}

Recall that a set $S\subseteq \I T^k$ is \emph{$r$-sparse} (given the free 
vectors 
$\V x_1,\dots,\V x_d$) if for every $\V u\in\I T^k$ and
every distinct $\V m,\V n\in S_{\V u}$ we have that $\|\V m-\V n\|_\infty> r$. 

For each $i\in\I N$, we construct a Borel set $S_i\subseteq \I T^k$ which is 
maximal $N_{i+2}$-sparse,
that is, $S_i$ is $N_{i+2}$-sparse but the addition of any new element of 
$\I T^k\setminus S_i$ to it violates this property. (The maximality property 
will be useful in the proof of Lemma~\ref{lm:VCDiam} as it will guarantee that the diameter 
of Voronoi cells of  $S_{i,\V u}$ is uniformly bounded.) Take the 
$N_{i+2}$-sparse map $\chi:\I T^k\to[t]$ provided by Lemma~\ref{lm:KST}.
We construct $S_i$ by starting with the empty set and then, iteratively for 
$j\in[t]$, adding all those points 
of $\chi^{-1}(j)$ that do not violate the $N_{i+2}$-sparseness with an already 
existing element. Formally, we let $S_{i,0}:=\emptyset$ and 
 $$
  S_{i,j}:=S_{i,j-1}\cup \left(\chi^{-1}(j)\setminus 
\cup_{\V n\in\dist_{\le N_{i+2}}(\B 0)}(S_{i,j-1}+n_1\V x_1+\dots+n_d\V 
x_d)\right),
   $$
 \mbox{for $j=1,\dots,t$}. This formula shows, in particular, that the final set $S_i:=S_{i,t}$ is Borel. Also, $S_i$ is $N_{i+2}$-sparse (since 
each $\chi^{-1}(j)$ is) while the maximality of $S_i$ follows from the fact 
that each element $\V x\in \I T^k$ 
was considered for inclusion into the set $S_{i,\chi(\V x)}\subseteq 
S_i$.

\subsubsection{Constructing grid domains around seeds}

Here we construct an equivariant family $(\C Q_{i,\V u})_{\V u\in \I T^k}$ consisting of disjoint $N_i$-cubes
in $\I Z^d$ that looks as the $N_i$-regular grid in a large neighbourhood 
of each point of $S_{i,\V u}$.  This construction is similar to the one by 
Tim\'ar~\cite{timar:04,timar:09:arxiv}, except he had to cover the whole space $\I Z^d$ 
with parts that could somewhat deviate from being perfect cubes.

 Let $\V u\in\I T^k$ and $\V s\in S_{i,\V u}$. Let the
\emph{(integer) Voronoi cell} of $\V s$ be
 \beq{eq:Csu}
  C_{i,\V s,\V u}:= \big\{\,\V n\in \I Z^d\mid \forall \V s'\in S_{i,\V u}\setminus\{\V s\}\ \ \|\V n-\V s\|_\infty<\|\V n-\V s'\|_\infty\,\big\},
  \eeq
   i.e.\ the set of points
in $\I Z^d$ strictly closer to $\V s$  than to any other element of~$S_{i,\V u}$.

Since each element of $\I Z^d$ is at distance at most $N_{i+2}$
from $S_{i,\V u}$, we can ``produce'' Voronoi cells using some $N_{i+2}$-local rule $\C R$.
(Namely, we want $\C R$ to transform the characteristic function of $S_i$ into the function
whose value on every $\V u\in \I T^k$ encodes if there is $\V s\in S_{i,\V u}$ such that 
$C_{i,\V s,\V u}$ contains the origin
and, if yes,  stores such (unique)  vector $\V s$.) In particular, the corresponding structure $C_i$ 
on $\I T^k$ is measurable by Lemma~\ref{lm:local}. 

Let $\C Q_{i,\V u}$ consist of those $N_i$-cubes $Q=\prod_{j=1}^d[a_j,a_j+N_i-1]$ 
for which there is $\V s\in S_{i,\V u}$ such that $Q\subseteq C_{i,\V s,\V u}$
and all coordinates of the vector $\V a-\V s$ are divisible by $N_i$. 
Since integer Voronoi cells are disjoint, the constructed cubes are also disjoint.
The
following lemma states, in particular, that the set of vertices missed by these cubes is ``small''.

\begin{lemma}\label{lm:VCDiam} Let $m\in\I N$ be arbitrary and, for $\V u\in\I T^k$, let  
$X_{\V u}\subseteq \I Z^d$ be the set of points at $L^\infty$-distance
at most $m$ from $\I Z^d\setminus \bigcup \C Q_{i,\V u}$. Then the (equivariant) family $(X_{\V u})_{\V u\in \I T^k}$
has uniform density at most $O((m+N_{i})/N_{i+2})$.\end{lemma}

\begin{proof}  Fix $\V u\in\I T^k$. Let 
 $$
 C_{i,\V s,\V u}^{\I R}:=\big\{\V x\in \I R^d\mid \forall \V s'\in S_{i,\V u}\ \ \|\V x-\V s\|_\infty\le \|\V x-\V s'\|_\infty\big\},
 $$
  be the \emph{real Voronoi cell} of $\V s\in S_{i,\V u}$. (The differences to the definition (\ref{eq:Csu})
are that now we consider any real vectors and we also include the boundary points.)

Take any $\V n\in X_{\V u}$. Let $\V s\in S_{i,\V u}$ be arbitrary such that the real
cube $[0,1)^d+\V n$ intersects $C_{i,\V s,\V u}^{\I R}$. We know that $\V n$ is at distance
at most $m$ from some $\V n'\in \I Z^d\setminus \bigcup \C Q_{i,\V u}$. Let $Q\subseteq \I Z^d$ be
the (unique) $N_i$-cube containing $\V n'$ that $\V s$ would have liked to claim (that is, $Q=\prod_{j=1}^d [a_j,a_j+N_i-1]\ni\V n'$ with each $a_j$ congruent to $s_j$ modulo $N_i$). Since $Q$ does not lie
inside the integer Voronoi cell of $\V s$, it has to contain a point $\V n''$
which is not farther from some $\V s'\in S_{i,\V u}\setminus\{\V s\}$ than from $\V s$. 
Thus $\dist(\V n,Y)\le \|\V n-\V n''\|_\infty\le 
m+N_i-1$, where $Y$ is the closure of $\I R^d\setminus C_{i,\V s,\V u}^{\I R}$. It follows that every element of 
$[0,1)^d+\V n$ is at 
distance at most $m+N_i$
from $Y$. 

On the other hand, by the $N_{i+2}$-spareness
of $S_{i,\V u}$ the distance between $\V s$ and $Y$ is larger than $N_{i+2}/2$.
Thus, if we shrink $C_{i,\V s,\V u}^{\I R}$ by factor $\gamma:=(N_{i+2}-2m-2N_i)/N_{i+2}$ from $\V s$, i.e.\
we take the set $\gamma\, C_{i,\V s,\V u}^{\I R}+(1-\gamma)\V s\subseteq C_{i,\V s,\V u}^{\I R}$, then it will be disjoint
from $[0,1)^d+\V n$. 
It follows that the set $[0,1)^d+ X_{\V u}:=\cup_{\V n\in X_{\V u}} ([0,1)^d+\V n)$ can cover at most $1-\gamma^d$
fraction of the volume of any Voronoi cell $C_{i,\V s,\V u}^{\I R}$. 

By the maximality of $S_{i,\V u}\subseteq \I Z^d$
(and since any point of $\I R^d$ is at distance at most $1/2$ from $\I Z^d$), 
the distance between $\V s\in S_{i,\V u}$ and any point on the boundary of $C_{i,\V s,\V u}^{\I R}$ 
is at most $N_{i+2}+1/2$. Thus the real Voronoi cells, that cover the whole space, 
have a uniformly bounded diameter. It follows that there is a constant $N$ (independent
of $\V u$) such that the
density of $X_{\V u}$ inside any $N$-cube is at most, say, 
$2(1-\gamma^d)=O((m+N_i)/N_{i+2})$, as required.\end{proof}

\subsubsection{Constructing the matchings $\C M_i$}

Iteratively for $i=0,1,\dots$, we will construct a 
measurable matching $\C M_i$ in $\C G$ (or, equivalently, 
an equivariant and measurable 
family $(\C M_{i,\V u})_{\V u\in\I T^k}$ where $\C M_{i,\V u}$ is a matching in 
$\C G_{\V u}$)
such 
that the following two properties hold for each $\V u\in\I T^k$:
 \begin{enumerate}
  \item every edge of $\C M_{i,\V u}$ lies inside some cube $Q\in\C Q_{i,\V u}$ (that is, 
$\C M_{i,\V u}\subseteq \cup_{Q\in \C Q_{i,\V u}} Q^2$);
   \item for every cube $Q\in\C Q_{i,\V u}$ the restriction of $\C M_{i,\V u}$ to $Q$
(more precisely, to the induced bipartite subgraph $\C G_{\V u}[Q,Q]$) 
is a maximum matching.
 \end{enumerate}

For $i=0$, we construct $\C M_{0,\V u}$ by taking a maximum matching inside each cube $Q\in\C Q_{0,\V u}$. In order to make it equivariant and measurable we consistently
use some local rule: for example, take the lexicographically smallest
maximum matching in $Q$, with respect to the natural labelling of the $N_0$-cube $Q$ by $[N_0]^d$. 

Let us explain how this can be realised by a local rule $\C R$ of radius $2N_0+N_2$. The rule transforms the function  $g:\I T^k\to\I N$ that encodes the triple of sets $(A,B,S_0)$ into
one that encodes $\C M_0$. Take any $\V a\in A$.
Note that the restriction of $g_{\V a}$ to $Q':=[-2N_0-N_2,2N_0+N_2]^d$ determines $A_{\V a}\cap Q'$,
$B_{\V a}\cap Q'$ and
$S_{0,\V a}\cap Q'$. Since $S_{0,\V a}$ is maximal $N_2$-sparse, we have that $S_{0,\V a}\cap [-N_2,N_2]^d\not=\emptyset$. From the latter set, take
a point $\V s$  which is closest to the origin~$\B 0$.
Let
$Q\subseteq\I Z^d$ be the $N_0$-cube that contains $\B 0$ and
belongs to the $N_0$-grid centred at~$\V s$. We have that
$Q\in\C Q_{0,\V a}$ if and only if every element of $Q$ is closer to $\V s$
than to  $S_{0,\V a}\setminus\{\V s\}$. We see that all elements of $S_{0,\V a}$ that
can ``interfere'' with~$Q$ are at
distance at most 
$N_0+N_2$ from $Q$ and thus are confined to~$Q'$.
Therefore, the set $S_{0,\V a}\cap Q'$ determines 
the $N_0$-cube of $\C Q_{0,\V a}$ containing~$\B 0$, if it exists (which has to 
be $Q$ then).
Suppose that $Q\in \C Q_{0,\V a}$. Since $Q$ is a subset of $Q'$, we know the intersections of
$A_{\V a}$ and $B_{\V a}$ with $Q$. Now, among all (finitely many) maximum 
matchings  in $\C G_{\V a}[A_{\V a}\cap Q,B_{\V a}\cap Q]$, take the  lexicographically 
smallest matching $\C M\subseteq Q^2$. If $\C M$ matches $\B 0\in A_{\V a}$,
then the local rule outputs $\V n:=\C M(\B 0)$, which tells us that the $\C M_0$-match of $\V a\in A$ is $\V a+\sum_{i=1}^d n_i\V x_i$. (This is an element of $B$ since $\V n\in B_{\V a}$.) Otherwise (namely, if $\V a\not \in A$, or $Q\not\in \C Q_{0,\V a}$, or
$\C M(\B 0)$ is undefined) the local rule outputs that $\V a$ is not matched.
The obtained matching
$\C M_0$ is measurable by Lemma~\ref{lm:local}, since each of the sets $A, 
B,S_0\subseteq \I T^k$ is measurable.
From now on, we may omit details like this.

Let $i\ge 1$ and suppose that we have $\C M_{i-1}$ satisfying all above properties. Let us describe how to construct $\C M_i$. We will do this in
three steps, producing intermediate matchings $\C M_i'$ and $\C M_i''$.
For each step, we also provide an upper bound on the measure of vertices 
in $A$ that undergo some change; these estimates will be later used to 
argue that~\eqref{eq:aim2} holds. So, take any $\V u\in \I T^k$. 

Let $\C M_{i,\V u}'$ consist of
those edges $(\V m,\V n)\in \C M_{i-1,\V u}$ for which there is a cube $Q\in\C Q_{i,\V u}$ with $\V m,\V n\in Q$. In
other words, when we pass from $\C M_{i-1}$ to $\C M_{i}'$, we discard all previous edges that have at least one vertex in 
$X:=\I T^k\setminus \bigcup \C Q_i$ or connect two different cubes of $\C Q_i$. By 
Lemma~\ref{lm:VCDiam}, the set $X$ has uniform density at most $O(N_i/N_{i+2})$ while, trivially,
points within distance $\MD$ from the boundary of some $\C Q_{i,\V u}$-cube
have uniform
density $O(1/N_i)$. Thus, by Lemma~\ref{lm:density}, when we pass from $\C M_{i-1}$ to $\C M_i'$, we change
the current matching on a set of measure 
 \begin{equation}\label{eq:Mi'}
 \lambda((\C M_{i-1}\bigtriangleup \C M_i')^{-1}(B))=O(N_i/N_{i+2}+1/N_i).
 \end{equation} 

Next, let $\C M_{i,\V u}''$ be obtained by modifying $\C M_{i,\V u}'$ as follows. 
For every cube $Q\in\C Q_{i,\V u}$ that has at least one vertex that lies outside of $\bigcup \C Q_{i-1,\V u}$ or has vertices that come from different integer Voronoi cells of $S_{i-1,\V u}$, let the restriction of
$\C M_{i,\V u}''$ to $Q$ be any maximum matching. (Thus we completely ignore $\C M_{i,\V u}'$ inside such
cubes $Q$.)
Lemma~\ref{lm:VCDiam} (applied to $i-1$ and $m=N_i$) 
shows that the union of such cubes $Q$ has uniform density $O(N_i/N_{i+1})$. Thus
by Lemma~\ref{lm:density}, we have that
 \begin{equation}\label{eq:Mi''}
  \lambda((\C M_{i}'\bigtriangleup \C M_i'')^{-1}(B))=O(N_i/N_{i+1}).
  \end{equation} 

Finally, we have to show how to obtain the desired matching $\C M_i$ by modifying
$\C M_i''$ on the remaining cubes, so that the new matching
is maximum inside each $\C Q_{i,\V u}$-cube. 
Let $Q\in\C Q_{i,\V u}$ be one of the remaining cubes. This means that $Q$ lies entirely inside
the Voronoi cell $C_{i,\V s,\V u}$ of some $\V s\in S_{i-1,\V u}$ and is completely covered by $N_{i-1}$-cubes
from $\C Q_{i-1,\V u}$. These cubes when restricted to $Q$ make a regular grid that, however, need
not be properly aligned with the sides of $Q$. Take $j\in[d]$. Let $I_j\subseteq \I Z$ be the projection of
$Q$ onto the $j$-th axis and let $I_{j,0}\cup \dots\cup I_{j,t_j}$ be the partition of $I_j$ into consecutive intervals given by the grid. Each $t_j$ is $2^h-1$ or $2^h$, where 
 $$
 h:=\log_2(N_i/N_{i-1}).
 $$ 
 (Since $N_{i-1}<N_i$ are powers of $2$, 
$h$ is a positive integer.)  If $t_j=2^h$, then we change the partition of $I_j$ as follows.
By reversing the order
of the second index if necessary, assume that $|I_{j,t_j}|\le N_{i-1}/2$ (note that $|I_{j,0}|+|I_{j,t_j}|=N_{i-1}$),
merge $I_{j,t_j}$ into $I_{j,t_j-1}$ (with the new interval still denoted as $I_{j,t_j-1}$),
and decrease $t_j$ by 1.
After we perform this merging for all $j$ with $t_j=2^h$, the obtained interval partitions induce a new grid on 
$Q$ that splits it into $2^{hd}$ rectangles. We call these rectangles \emph{basic} and naturally index them by $(\{0,1\}^d)^h$. Namely,  for
$\V B=(\V b_1,\dots,\V b_h)$ with each $\V b_t$ being a binary sequence of length $d$,
let $R^{\V B}:=\prod_{j=1}^d I_{j,s_j}$, where $s_j\in[0,2^h-1]$ is the integer whose 
base-$2$ representation has digits $(b_{1,j},\dots,b_{h,j})$. Furthermore, if $(\V b_1,\dots,\V b_t)\in
(\{0,1\}^d)^t$ with $t< h$, then we define 
 $$
 R^{(\V b_1,\dots,\V b_t)}:=\cup_{(\V b_{t+1},\dots,\V b_h)\in  (\{0,1\}^d)^{h-t}}\, R^{(\V b_1,\dots,\V b_h)}
  $$
 to be the union of  those basic rectangles whose index sequence has $(\V b_1,\dots,\V b_t)$ as a prefix.
Thus,
the largest (or \emph{level-0}) rectangle is $R^{(\,)}=Q$, which splits into $2^d$ \emph{level-1} rectangles as $R^{(\,)}=\cup_{\V b\in \{0,1\}^d} R^{(\V b)}$, each of which splits further as $R^{(\V b)}=\cup_{\V a\in \{0,1\}^d}\,  R^{(\V b,\V a)}$, and so on
until we get the basic rectangles at \emph{level~$h$}. See Figure~\ref{fg:rectangles} for an
illustration with $d=h=2$.

\begin{figure}[t]
\begin{center}
\includegraphics[height=3.7cm]{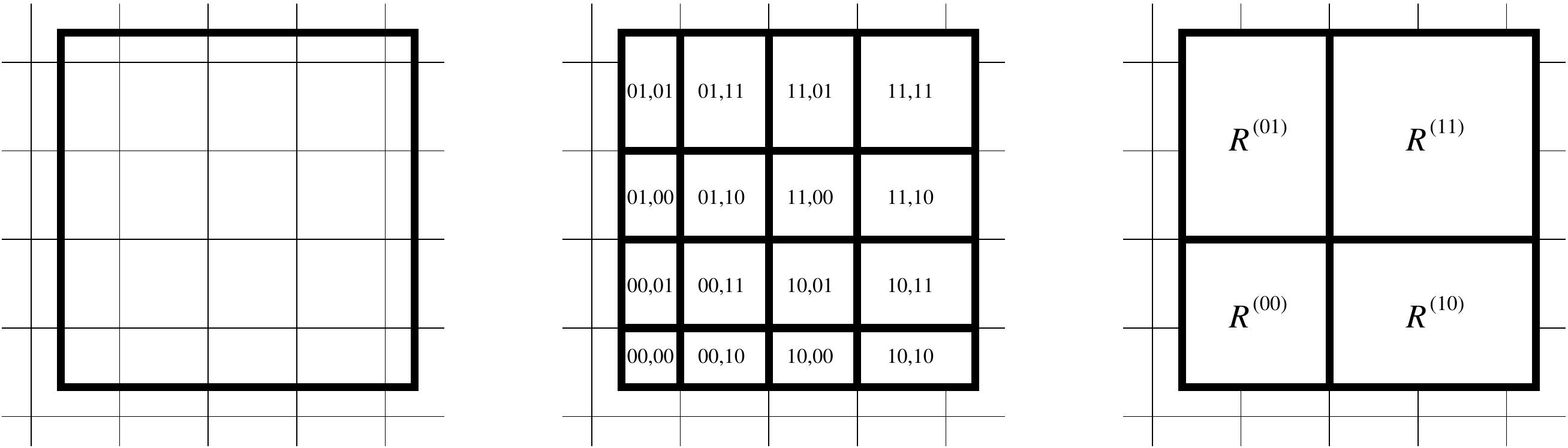}
\end{center}
\caption{\quad i) Cube $Q$ and $\C Q_{i-1,\V u}$-grid;\quad ii) basic rectangles;\quad iii) level-1 rectangles}\label{fg:rectangles}
\end{figure}

Note that, for each $j\in [0,h]$, a level-$j$ rectangle has side lengths between $2^{h-j}N_{i-1}\pm N_{i-1}/2$;
in particular it is $3$-balanced.  Also,
$\C M_{i,\V u}''\subseteq \C M_{i-1,\V u}$ cannot connect two different basic rectangles. Since $\C M_{i,\V u}''\cap Q^2=\C M_{i-1,\V u}\cap Q^2$, it is a maximum matching
on every basic rectangle which is  an element of $\C Q_{i-1,\V u}$. 

We are now ready to describe how we modify $\C M_{i,\V u}''$ on $Q$. First, put an
arbitrary maximum matching on every basic rectangle not in $\C Q_{i-1,\V u}$. 
(For example, in Figure~\ref{fg:rectangles}.ii these happen to be all 12 rectangles at the boundary of~$Q$.)
This involves at most $2N_{i-1}\cdot dN_i^{d-1}$ elements of $Q$ and thus has uniform density at most $O(N_{i-1}/N_i)$. 
Next, we iteratively repeat the following for $\ell=h,\dots,1$. Suppose that the current matching is maximum when restricted to each level-$\ell$
rectangle (i.e.\ to each $R^{(\V b_1,\dots,\V b_{\ell})}$) and does not connect two such rectangles. 
(Note that this is the case at the initial step $\ell=h$.)
Inside each level-$(\ell-1)$ rectangle $R$, iteratively augment the current matching using paths of
length at most $ (2^{h-\ell+1}+1/2)N_{i-1}$ until none remains. Clearly, each augmentation
increases the size of the matching inside the finite set $R$, so we run out
of augmenting paths after finitely many flips. 
By Lemma~\ref{lm:ShortAug}, the final matching in $\C G_{\V u}[R,R]$ covers all 
vertices in one part. In particular, it is maximum (and we can proceed with the 
next value of $\ell$). 

Note that the number of unmatched
vertices in any level-$(\ell-1)$ rectangle $R^{\V B}$, $\V B\in (\{0,1\}^d)^{\ell-1}$, before Iteration $\ell$ was
at most  $\sum_{\V b\in \{0,1\}^d}
D(R^{(\V B,\V b)})$, where
 \beq{eq:D}
 D(Y):=\big|\, |A_{\V u}\cap Y|-|B_{\V u}\cap Y|\, \big|
 \eeq
 denotes the \emph{$(A_{\V u},B_{\V u})$-discrepancy} of a finite set $Y\subseteq \I Z^d$. 
(This holds because, again by Lemma~\ref{lm:ShortAug}, a maximum matching
inside $R^{(\V B,\V b)}$ for each $\V b\in \{0,1\}^d$ has to match one part completely.) Thus, when we pass from $\C M_i''$ to
$\C M_i$, the density of changes inside $Q$ caused by augmenting paths is at most
 \beq{eq:QChange}
  \frac{1}{|Q|}\sum_{\ell=1}^h\, O(N_i/2^{\ell}) \sum_{\V B\in(\{0,1\}^d)^{\ell}} D(R^{\V B}).
 \eeq

Call a level-$\ell$ rectangle $R$ \emph{special} if at least two of its side lengths are different 
from $N_{i}/2^\ell$. (For example, in Figure~\ref{fg:rectangles}.ii the special rectangles happen
to be all four corner rectangles.)

The $(A_{\V u},B_{\V u})$-discrepancy of a non-special level-$\ell$ rectangle $R$ can be bounded by decomposing it
into $(d-1)$-dimensional $N_{i}/2^\ell$-cubes and using the $\Psi$-uniformity of both $A_{\V u}$
and $B_{\V u}$:
 $$
 D(R)\le 2\,\Psi(N_{i}/2^\ell)\, \frac{|R|}{(N_{i}/2^\ell)^{d-1}}.
 $$
 Thus such rectangles contribute $O(\Psi(N_{i}/2^\ell)/(N_{i}/2^\ell)^{d-2})$ to (\ref{eq:QChange}).

We bound the
$(A_{\V u},B_{\V u})$-discrepancy of a special level-$\ell$ rectangle $R$  using the following argument with $N:=N_{i}/2^\ell$
and $n:=N_{i-1}/2$. For notational convenience, assume that $R=\prod_{j=1}^d [r_j]$ 
(thus $|r_j-N|\le n$ for each $j\in[d]$) and
that $r_j\ge N$ exactly for $j\in [t]$. For $j\in[t]$ (resp.\ $j\in[t+1,d]$)
let $R_j$ be the rectangle which is the product of $d$ copies of $[N]$ except the $j$-th factor
is $[N+1,r_j]$ (resp.\ $[r_j+1,N]$).  We can transform $[N]^d$
into $R$ by adding the rectangles $R_1,\dots,R_{t}$, then subtracting $R_{t+1},\dots,R_d$, and finally
adjusting those vertex multiplicities that are still wrong. The last step involves at most ${d\choose 2} n^2N^{d-2}$ vertices of $R$ and each multiplicity has to be adjusted by at most~$d$. Thus we have the
following Bonferroni-type inequality:
 $$
 D(R)\le D([N]^d)+ \sum_{j=1}^d D(R_j) + d {d\choose 2} n^2N^{d-2}
  = O(N\,\Psi(N)+ n^2N^{d-2}).\label{eq:D2}
 $$
 
Also, the total number of special rectangles at level $\ell$ is at most 
 $O(2^{\ell(d-2)})
 $.
 Indeed,  we have at most $4{d\choose 2}$ ways  to choose a ``$(d-2)$-dimensional face'' 
$F$ of $Q$, and then observe that $F$ intersects
at most $2^{\ell(d-2)}$ level-$\ell$ rectangles (while each special rectangle 
must have a non-empty footprint in at least one $(d-2)$-dimensional face of $Q$).

Thus, special level-$\ell$ rectangles contribute $O\big(4^{-\ell}\,\Psi(N_i/2^\ell)/(N_i/2^\ell)^{d-2}+2^{-\ell} N_{i-1}^2/N_i\big)$ to~(\ref{eq:QChange}). 

Putting all together and using Lemma~\ref{lm:density}, we get the following upper bound on the measure of 
points where $\C M_i''$ and $\C M_i$ differ: 
 \begin{equation}\label{eq:Mi'''}
 \lambda((\C M_i''\triangle\C M_i)^{-1}(B))
 =    O(N_{i-1}^2/N_i)+  O(1)\, \sum_{\ell=1}^h \frac{\Psi(N_{i}/2^\ell)}{(N_{i}/2^\ell)^{d-2}}.
 \end{equation}
 
Having constructed $\C M_i$, we increase $i$ by one and repeat the above procedure. It remains to show that the constructed sequence of measurable matchings $(\C M_i)_{i\in \I N}$ has the required properties. 

Observe that each upper bound in~\eqref{eq:Mi'}, \eqref{eq:Mi''}, and~\eqref{eq:Mi'''} is a summable
function of $i\in\I N$. This directly follows from our choice of the sequence $(N_i)_{i\in\I N}$, with the 
exception of the second term in the right-hand side of~\eqref{eq:Mi'''}. Here, if we sum these terms over $i\in\I N$, then each integer power $2^j$ can appear as $N_i/2^\ell$ at most once 
(indeed, $i\in\I N$ has to be the unique index such that $N_{i-1}\le 2^j<N_{i}$);
thus the
resulting sum converges by (\ref{eq:SumPsiFin}). 
Now (\ref{eq:aim2}) follows, since
 $$
  \C M_{i-1}\bigtriangleup\C M_i\subseteq (\C M_{i-1}\bigtriangleup\C M_i')\cup 
  (\C M_{i}'\bigtriangleup\C M_i'')\cup
  (\C M_{i}''\bigtriangleup\C M_i).
  $$

Since $\C M_{i,\V u}$ covers all but at most $D(Q)\le 2\Phi(N_i)$ 
vertices inside each cube in $\C Q_{i,\V u}$ while the set of vertices not covered by $\C Q_{i}$
has uniform density $O(N_i/N_{i+2})$ by Lemma~\ref{lm:VCDiam}, the set $A\setminus\C M_i^{-1}(B)$ has
uniform density at most $O(\Phi(N_i)/N_i^d+N_i/N_{i+2})$. Clearly, this tends to zero as $i\to\infty$. Thus,
by Lemma~\ref{lm:density}, the other desired estimate (\ref{eq:aim1}) also holds.

The proof of Part~\ref{it:Lebesgue} of Theorem~\ref{th:Suff}
can now be completed, as it was described
after (\ref{eq:aim2}).

\subsection{Proof of Lemma~\ref{lm:ShortAug}}\label{ShortAug}

This section is dedicated to proving 
Lemma~\ref{lm:ShortAug} that was needed in the proof of Part~1 of 
Theorem~\ref{th:Suff}.  
Note that Lemma~\ref{lm:ShortAug}
is a combinatorial statement that does not involve any notion of equivariance or 
measurability.

For $X,R\subseteq \I Z^d$,  define the
 \emph{internal boundary of $X$ relative to $R$} to be 
 $$
 \partial^R X:=(\partial X)\cap R^2=\{(\V m,\V n)\in R^2\mid \V m\in X,\ \V n\not\in X\}
 $$ 
 and the \emph{internal perimeter of $X$ relative to $R$} to be $p^R(X):=|\partial^R X|$. Recall that a rectangle $R\subseteq \I Z^d$ is called  
\emph{$\rho$-balanced} if the ratio of any its two side lengths is at most $\rho$.

Our next lemma states that a positive fraction of the boundary of a set $X$ lying inside a $\rho$-balanced
rectangle $R$ is internal, unless $X$ occupies most of~$R$.

\begin{lemma}\label{lm:boundaries} For every $R,X\subseteq \I Z^d$,
where $X\subseteq R$ and $R$ is a $\rho$-balanced rectangle,
we have that 
 $$
  p^R(X)\ge \frac{p(X)}{3d\rho}\cdot  \frac{ |R\setminus X|}{|R|}.
  $$
\end{lemma}

\begin{proof} Let $\gamma:=d\rho/2+2d+1$, which is at most $3d\rho$ since $d\ge 2$ and $\rho\ge1$. We can assume that
$p^R(X)<p(X)/\gamma$ for otherwise we are trivially done. Let $j\in [d]$
be such that at least $1/d$-fraction of the common boundary
$\partial X\cap \partial R$ goes in the $j$-th coordinate direction. 
Thus the number of lines parallel to the $j$-th coordinate axis that 
intersect $X$ is at least half of this quantity. Trivially, at most
$p^R(X)$ of these lines can intersect $R\setminus X$ while every other line
contains $r_j$ points from $X$, 
where $r_1,\dots,r_d$ are the side lengths of $R$. Thus
 \begin{equation}\label{eq:|X|}
  |X|\ge r_j\, \left(\frac{p(X)-p^R(X)}{2d}-p^R(X)\right)\ge \frac{r_j(\gamma-2d-1)}{2d\gamma}\, p(X).
 \end{equation}

For every pair $(\V a,\V b)\in X\times (R\setminus X)$ consider the path inside $R$
which is the union of the $d$ ``straight-line'' paths that connect 
the following $d+1$ points in the stated order:
 $$
  \V a=(a_1,\dots,a_d),\ \ (b_1,a_2,\dots,a_d),\ \ (b_1,b_2,a_3,\dots,a_d),\ \ \dots,\ \
 \V b=(b_1,\dots,b_d).
 $$ 
 Each such path contains at least one pair from $\partial^RX$. Conversely,
every (ordered) pair in the internal boundary of $X$ going, for example, in the 
$i$-th direction
is in at most $(r_i^2/4) \prod_{h\in [d]\setminus\{i\}} r_h=(r_i/4)|R|$ such 
paths. Denoting $r:=\max\{r_i\mid i\in [d]\}$ and using~\eqref{eq:|X|}, we 
conclude that
 $$
 p^R(X)\ge \frac{|X|\cdot|R\setminus X|}{(r/4)\, |R|}\ge \frac{2(\gamma-2d-1)p(X)}{d\rho \gamma}\cdot \frac{|R\setminus X|}{|R|}.
 $$
 Using that $\gamma\le 3d\rho$ satisfies $2(\gamma-2d-1)=d\rho$, we obtain the required bound.\end{proof}

For a real $\delta\ge 0$ and sets $X,R\subseteq \I Z^d$ such 
that $R$ is finite, let the \emph{discrepancy of $X$ relative to $R$
of density $\delta$} be
 $$
 D_{\delta}(X;R):=\big|\,|X\cap R|-\delta\, |R|\,\big|.
 $$
 Thus,
 a set $X\subseteq\I 
Z^d$ is $\Phi$-uniform of density $\delta$ if and only if $D_\delta(X;Q)\le 
\Phi(2^i)$ for 
every $2^i$-cube $Q\subseteq \I Z^d$.

\renewcommand{\CC}{M}

Clearly, Lemma~\ref{lm:ShortAug} follows from Part~2 of the following result
when applied to $\CA:=A_{\V u}$ and $\CB:=B_{\V u}$,
assuming that $\MD$ satisfies 
Lemma~\ref{lm:ShortAugAll} for $\rho:=3$ 
(with $d\in\I N$, $\delta>0$, and $\Phi(2^i)=2^i\Psi(2^i)$ being as in 
Theorem~\ref{th:Suff}).

\begin{lemma}\label{lm:ShortAugAll} For every integer $d\ge 1$, reals
$\delta>0$ and $\rho\ge 1$, and a function $\Phi:\{2^i\mid i\in\I N\}\to \I R$ 
satisfying $\sum_{i=0}^\infty \Phi(2^i)/2^{(d-1)i}<\infty$,
there is a constant $\CC=\CC(d,\delta,\rho,\Phi)$ such that the following holds 
for every
$\rho$-balanced rectangle $R:=\prod_{j=1}^d 
[a_j,a_j+r_j-1]\subseteq \I Z^d$ and sets $\CA,\CB\subseteq \I Z^d$ that are 
$\Phi$-uniform of density $\delta>0$. 

If $N:=\max(r_1,\dots,r_d)$ is the maximum side length of $R$ and 
 $$\C F:=(\CA\cap R,\linebreak\CB\cap R,E)$$ is the bipartite graph with edge set
 $
 E:=\{(\V m,\V n)\in (\CA\cap R)\times (\CB\cap R)\mid \|\V m-\V n\|_{\infty} \le \CC\},
 $ 
then
 \begin{enumerate}
 \item\label{it:StrongHall} for every $X\subseteq \CA\cap R$, we have 
 $$
 |E(X)|\ge \min\left(|X|+10\,d\cdot |X|^{(d-1)/d},\ \frac{|\CB\cap R|}2\right),
 $$
 where $E(X)=\{y\mid \exists\, x\in X\ (x,y)\in E\}$ is the neighbourhood 
of $X$ in~$\C F$;
 \item\label{it:ShortAug} for every matching $\C M$ in $\C F$ that leaves 
unmatched elements
in both parts, there is an augmenting path of length at most $N$.
 \end{enumerate}
\end{lemma}

\begin{proof} Given $d$, $\delta$, $\rho$, and $\Phi$, choose sufficiently large integers 
in the order 
$\CC_0\ll \CC_1\ll\CC$. 

Since $\sum_{i=0}^\infty \Phi(2^i)/2^{(d-1)i}<\infty$,
Theorem~1.2 in  Laczkovich~\cite{laczkovich:92}  shows 
that, for every $X\subseteq \I Z^d$ which is $\Phi$-uniform of density $\delta$, 
 we have
 \beq{eq:DAlpha}
 D_\delta(X;Y)\le \CC_0\, p(Y),\quad\mbox{for all finite $Y\subseteq\I Z^d$}.
 \eeq
  Note that 
the coefficient $\CC_0=\CC_0(d,\delta,\rho,\Phi)$ in (\ref{eq:DAlpha}) 
does not depend on the choice of~$X$ and~$Y$. (The proof of~\eqref{eq:DAlpha} 
in~\cite{laczkovich:92} proceeds 
by
representing each $Y$ as a certain combination of binary cubes having the 
appropriately defined ``complexity'' at most 
$\CC_0\, p(Y)$.)
 \hide{
 In brief, the proof of (\ref{eq:DAlpha}) proceeds as follows. Theorem 
1.3 in
\cite{laczkovich:92} shows  that  for every finite $Y\subseteq \I Z^d$  
there are cubes $Q_1,\dots,Q_n$ such that, for every $j\in\I N$, this sequence 
has at most  $\CC_1 p(Y)/2^{(d-1)j}$
$2^j$-cubes, and we can recursively construct $Y$ from these cubes where at 
each step we take the union of two disjoint sets or
subtract a subset from a set. Furthermore, each cube $Q_j$ is used at most once; thus it follows
that 
 $$
  D_{\delta}(X;Y) \le \sum_{t=1}^n D_{\delta}(X;Q_t)\le \sum_{j=0}^\infty 
\frac{\CC_1p(Y)}{2^{(d-1)j}}\cdot \Phi(2^j)
\le \CC_0\, p(Y).
 $$
 }

Next, take arbitrary $\CA,\CB,R\subseteq \I Z^d$ as in the statement of the lemma. 
Assume that $N>\CC$ for otherwise $\C F$ is a complete bipartite graph and the lemma
trivially holds.

We prove Part~\ref{it:StrongHall} of the lemma first. Fix an arbitrary 
set 
$X\subseteq \CA\cap R$ with $2\, |E(X)|< |\CB\cap R|$.

Very roughly, the proof proceeds as follows. We define
some ``smoothed'' versions $X_1$ and $X_2$ of $X$, where for illustration
purposes one can imagine $X_i$ as $\dist_{\le i\CC_1}(X)\subseteq 
\I Z^d$, the $(i\CC_1)$-ball around $X$.  Then every point of $\CB\cap X_2$ is 
a neighbour of $X$. Also, if the boundaries of $X_1$ and $X_2$ are ``smooth
on the scale of $\CC_1$'', then we expect that $|X_2\setminus X_1|\ge \Omega\big(\CC_1 
(p(X_1)+p(X_2))\big)$. On the other hand, by~\eqref{eq:DAlpha} the number of points
from $\CA$ and $\CB$ inside $X_i$
deviates from the expected value $\delta|X_i|$ by at most $\CC_0\,p(X_i)$
which is much smaller than~$\CC_1\, p(X_i)$. This suffices 
to get the desired gap between $|E(X)|\ge |\CB \cap X_2|$ and 
$|X|\le |\CA\cap X_1|$. The above argument is from 
Laczkovich~\cite{laczkovich:92}.
However, here we have to confine all sets to $R$ (even if $X$ comes 
very close
to the boundary of~$R$). With an appropriate definition of $X_i$, 
the above argument can be applied if
$|X_i|<\frac23 |R|$ as then
$X_i$ has a positive fraction of its boundary in the interior of $R$ 
by Lemma~\ref{lm:boundaries}. Otherwise there is a large discord
between $|X_i|\ge \frac23 |R|$ and $|\CB\cap X_i|\le |E(X)|< 
|\CB\cap R|/2$, implying that $p(X_i)=\Omega(N^d)$; thus $p^R(X_i)\ge p(X_i)-p(R)=p(X_i)-O(N^{d-1})$
is essentially the same as $p(X_i)$ and the above argument still applies.
Let us give all details 
now.

For each $j\in [d]$, fix a partition $[a_j,a_j+r_j-1]=\cup_{i=1}^{t_j} I_{j,i}$ 
of the $j$-th side of $R$ into intervals of length $\CC_1$ and $\CC_1+1$. 
(For example, take 
$r_j\pmod{\CC_1}$ intervals of length $\CC_1+1$ that occupy at most 
$(\CC_1-1)(\CC_1+1)\le \CC/\rho\le r_j$ initial elements 
and split the rest into
intervals of length~$\CC_1$.)  Call each $d$-dimensional product 
$\prod_{j=1}^d I_{j,s_j}$ with $\V s\in \prod_{j=1}^d [t_j]$ a \emph{sub-rectangle}. Thus we have partitioned $R$ into an
(almost regular) grid made of sub-rectangles. We say that two sets $Y,Z\subseteq \I Z^d$ \emph{share boundary} if there is
$(\V m,\V n)\in \partial Y$ such that $(\V n,\V m)\in \partial Z$. By the grid structure,
each sub-rectangle can share boundary with at most $2d$ other sub-rectangles.

Let $X_1$ be the union of all sub-rectangles
that intersect $X$.  Let 
$X_2$ be obtained from $X_1$ by adding all sub-rectangles that share boundary with it.
Clearly, 
 $$X\subseteq X_1\subseteq X_2\subseteq R.
 $$
By Lemma~\ref{lm:Isop}, we have that $p(X_2)\ge2d\cdot |X_2|^{(d-1)/d}\ge 2d\cdot |X|^{(d-1)/d}$. Thus, in order to prove
the first part of the lemma, it is enough to show that 
 \beq{eq:aim3}
 |E(X)|\ge|X|+5\,p(X_2).
 \eeq

 First, let us show that 
 \beq{eq:B-A}
 |X_2\setminus X_1| \ge \frac{\CC_1}{2d}\, \left(\,p^R(X_1)+p^R(X_2)\,\right).
 \eeq
 Note that if $(\V n,\V n+\V e)$ is in $\partial^R X_1$ (resp.\
$\partial^R(R\setminus X_2)$),  then $\V n+ i\V e\in X_2\setminus X_1$
for all $i\in [\CC_1]$. (Indeed, the directed edge $(\V n,\V n+\V e)$ enters 
some sub-rectangle $R'\subseteq X_2\setminus X_1$ and it takes at least $\CC_1$ 
steps in that direction
before we leave~$R'$.) This way we encounter 
at least $\CC_1\,p^R(X_1)$ (resp.\
$\CC_1\,p^R(X_2)$) elements in $X_2\setminus X_1$ with each element counted
at most $2d$ times in total, which gives~(\ref{eq:B-A}).

Since we may assume that $2(\CC_1+1)\le \CC$, each element of $X_2$ is within 
distance $\CC$ from $X$;
thus $E(X)\supseteq \CB\cap X_2$. Also, by the construction of $X_1$, we have $X\subseteq \CA\cap X_1$. We
conclude by (\ref{eq:DAlpha})  and (\ref{eq:B-A}) that 
 \begin{eqnarray}
 |E(X)|-|X|&\ge& |\CB\cap X_2|-|\CA\cap X_1|\nonumber\\
  &\ge& \delta\, |X_2\setminus X_1| - \CC_0\big(p(X_1)+p(X_2)\big)\label{eq:HallDiff}\\
   &\ge &  \frac{\delta \CC_1}{2d}\, \big(p^R(X_1)+p^R(X_2)\big)
   - \CC_0\big(p(X_1)+p(X_2)\big).\nonumber
 \end{eqnarray}

Let $i=1$ or $2$. Let us show that $p^R(X_i)\ge p(X_i)/\CC_0$. 
This directly follows from Lemma~\ref{lm:boundaries}
if $|X_i|\le \frac23\, |R|$ since we can assume that $\CC_0\ge 9d\rho$.
So suppose that $|X_i|> \frac23\, |R|$. Since $\CB\cap X_i\subseteq E(X)$ has, by our
assumption, less than $|\CB\cap R|/2$ elements, we have by~\eqref{eq:DAlpha} applied twice to the 
$\Phi$-uniform set $\CB$ 
that
 \begin{eqnarray*}
 \CC_0\,p(X_i)
 &\ge& \delta|X_i|-|\CB\cap X_i|\ \ge\ \frac{2\delta|R|}{3}- \frac{|\CB\cap R|}2\\
 &\ge&  \frac{2\delta|R|}{3}-\frac{\delta|R|+\CC_0\,p(R)}2\ =\ \frac{\delta|R|}6-\frac{\CC_0\,p(R)}2.
 \end{eqnarray*}
 Since $|R|\ge (N/\rho)^d$ and $p(R)\le 2dN^{d-1}$ (and $N> \CC\gg \CC_0$), we conclude that 
 $$
  p^R(X_i)\ge p(X_i)-p(R)\ge p(X_i)/2,
 $$
 which is even stronger than claimed.

Now, by 
(\ref{eq:HallDiff})
and since $\CC_1\gg \CC_0$, we obtain the required
bound (\ref{eq:aim3}):
  $$
   |E(X)|-|X|\ge \left(\frac{\delta 
\CC_1}{2d\CC_0}-\CC_0\right)\big(p(X_1)+p(X_2)\big)\ge 5\,p(X_2).
   $$
  This finishes the proof of Part~1.

Now, we prove the second part of the lemma, assuming the validity of 
Part~\ref{it:StrongHall}.

Let $\CA_{0}\not=\emptyset$ consist of the unmatched points in $\CA\cap R$.  
Define an \emph{alternating path} as a sequence $(x_0,\dots,x_\ell)$ such that 
$x_0\in \CA_{0}$, $(x_{i-1},x_i)\in E\setminus 
\C M$ for all odd $i\in[\ell]$ and $(x_i,x_{i-1})\in \C M$ for all even 
$i\in[\ell]$ (i.e.\ it is a path in $\C F$ that
starts with an unmatched vertex of $\CA\cap R$ and alternates between unmatched 
and matched edges). 
Note that any alternating path that ends in an unmatched vertex is augmenting.
For $i\in\I N$, let $\CA_{i}$ be the set of endpoints of alternating paths 
whose length is at most $i$ and has the same parity
as $i$. (Thus $\CA_{i}\subseteq \CA\cap R$ for even $i$ and $\CA_{i}\subseteq 
\CB\cap R$ for odd $i$.)
Let $\CA_{i}'$
consist of vertices reachable by an alternating path of length $i$ but not by a 
shorter one (that is, 
$\CA_{i}':=\CA_{i}\setminus \CA_{i-2}$ for $i\ge 2$ and 
$\CA_{i}':=\CA_{i}$ for $i=0,1$).

Suppose on the contrary that there is no augmenting path of length at most $N$. 
Take an odd integer $\ell$
such that $N-3\le 2\ell+1\le N$.

Roughly, the proof proceeds as follows. It is easy to see that 
$\CA_{2i+1}=E(\CA_{2i})$ for
all $i\in \I N$ and the absence of short augmenting paths implies that $\C M$ 
covers 
all of $\CA_{2i+1}$ for $i\le \ell$. It follows from Part~1 that $|\CA_{i}|$ 
grows as $\Omega(i^d)$. By symmetry, the same applies when we grow alternating 
paths starting from~$\CB$. These two
processes have to collide in fewer than $N/2$ steps, giving a contradiction. Let us 
provide the details now. 

Clearly, if $i\le \ell$, then $\CA_{2i+1}'\subseteq \C M(\CA)$ (i.e.\ every 
vertex of $\CA_{2i+1}'\subseteq
\CB\cap R$ is matched) for otherwise we have a too short augmenting path. 
Furthermore, $\CA_{2i+2}'=\C M^{-1}(\CA_{2i+1}')$ (i.e.\
$\C M$ gives a bijection from $\CA_{2i+2}'$ to $\CA_{2i+1}'$).
Thus, by induction on $i$, 
we have that
 \beq{eq:XSizes} 
 |\CA_{2i+2}|=|\CA_{2i+1}|+|\CA_{0}|,\quad \mbox{for every $i\le \ell$}.
 \eeq

Next, let us show that 
 \beq{eq:Xell}
 |\CA_{\ell}|\ge |\CB\cap R|/2.
 \eeq

We may assume that $|\CA_{i}|< |\CB\cap R|/2$ for all odd $i\le \ell$, as 
otherwise we are done
since $\CA_{\ell}\supseteq \CA_{i}$. Take any integer $i\in[0,(\ell-1)/2]$.
Since $\CA_{2i+1} \supseteq E(\CA_{2i})$ (in fact, this is equality),
we have by Part~1
of the lemma that
$$
 |\CA_{2i+1}|\ge|E(\CA_{2i})|\ge |\CA_{2i}|+10\,d\cdot |\CA_{2i}|^{(d-1)/d}.
 $$ 
 If $i\ge 1$, then  (\ref{eq:XSizes}) implies that $|\CA_{2i}|\ge 
|\CA_{2i-1}|$. Thus the sizes 
$a_i:=|\CA_{2i-1}|$
satisfy $a_1\ge 1$ and 
 \begin{equation}\label{eq:aineq}
  a_{i+1}\ge a_i+10\,d\cdot a_{i}^{(d-1)/d},\quad\mbox{for all 
$i\in[(\ell-1)/2]$.}
 \end{equation}
 Let an integer $\gamma=\gamma(d)$ be such that $\gamma^{j-1}\ge 5^{j-1}{d\choose j}$ for all $j\in [2,d]$; for example, we can take $\gamma=5\cdot 2^d$.
 We obtain that $a_{i+1}\ge a_i+1$ and, if $a_i\ge c^d$ 
with $c\ge \gamma$ then we have by~\eqref{eq:aineq} 
 $$
 a_{i+1}\ge c^d+10\,d\cdot c^{d-1} \ge c^d+ 5\,d\cdot c^{d-1}+ \sum_{j=2}^d 
5\,c^{d-1}
\ge \sum_{j=0}^d {d\choose j}\,5^j\,c^{d-j}=(c+5)^d.
 $$
 
  We conclude by induction on $i$  that
 $a_{i+\gamma^d}\ge (5\,i+\gamma)^d$ for every 
 $i\in[0,(\ell-1)/2-\gamma^d]$. But 
 then $a_i$ is larger than $N^d\ge |R|$ for some $i\le 
 N/5+\gamma^{d}\le(\ell-1)/2$. (Recall
 that $d\ll \CC<N\le 2\ell+4$.)
This contradiction proves~(\ref{eq:Xell}).

Likewise, by swapping the role of the sets $\CA$ and $\CB$ in the previous 
argument
(and extending the definition of an alternating path accordingly), 
let $\CB_{i}$ consist of vertices reachable by alternating paths that start in 
an unmatched vertex of $\CB\cap R$ and whose length is at most
$i$ and of  the same parity as $i$. As above, we conclude that 
$|\CB_{\ell}|\ge |\CA\cap R|/2$. Assume without loss of generality that 
$|\CB\cap R|\ge |\CA\cap R|$.
By (\ref{eq:XSizes}) and (\ref{eq:Xell}), we have that 
 $$|\CA_{\ell+1}|=|\CA_{\ell}|+|\CA_{0}|> |\CA_{\ell}|\ge |\CB\cap R|/2\ge 
|\CA\cap R|/2.
  $$ 
Thus $\CB_{\ell}$, which occupies at least half of $\CA\cap R$, intersects 
$\CA_{\ell+1}$. 

Take two alternating paths $P:=(x_0,\dots,x_s)$ and $P':=(y_0,\dots,y_t)$
such that $x_0\in \CA$, $y_0\in \CB$, $x_s=y_t$, and $s+t$ is smallest possible;
if there is more than one choice, let $s$ be maximum. Clearly, $s+t$ is odd 
and, by
above, it is at most $2\ell+1\le N$. Thus we cannot have $t=0$ for otherwise 
$P$ is a short augmenting path 
contradicting our assumption. Since $s$ was
maximum, we cannot append $y_{t-1}$ to $P$ and still have an alternating path. 
By the parity of $s+t$, the only possible reason is that
$y_{t-1}$ is equal to some $x_j$ with $s-j$ being a positive odd integer. But 
then $(x_0,\dots,x_{j})$ and $(y_0,\dots,y_{t-1})$ are alternating paths that 
contradict the choice of the pair $(P,P')$, namely,
the minimality of $s+t$.
This final contradiction proves Lemma~\ref{lm:ShortAugAll}.\end{proof}

\section{Proof of Part~\ref{it:Baire} of  Theorem~\ref{th:Suff}}\label{Baire}

In this section, let \emph{measurable} mean Baire measurable.

Given 
the function $\Phi:\{2^i\mid i\in\I N\}\to\I R$, 
the vectors $\V x_1,\dots,\V x_d\in \I T^k$, and the 
sets $A,B\subseteq \I 
T^k$
as in Part~\ref{it:Baire} of Theorem~\ref{th:Suff}, 
let $\MD$ be sufficiently large and then let 
$r_1<r_2<\dots$ be
a fast-growing sequence of integers. (Namely,  for each~$i$, the constant $\MD$ and the subsequence $(r_1,\dots,r_i)$ 
have to satisfy 
Lemma~\ref{lm:Baire}.) 

Given our choice of $\MD$, consider the usual bipartite graph $\C G=(A,B,E)$
as defined in~\eqref{eq:CG}; namely, 
$E=\{(\V a,\V b)\in A\times B\mid \V a-\V b\in\VV{\V x_1,\dots,\V x_d}{\MD}\}$. Thus we would like to find a measurable
perfect matching 
$\C M$ in $\C G$.
We will use essentially the same construction
that appears in Marks and Unger~\cite{marks+unger:16}. However, the 
analysis of its
correctness for our problem is much more complicated. 

Recall that a set $X\subseteq \I T^k$ is called \emph{$r$-sparse} if no 
$X_{\V u}$ contains two  distinct vectors at $L^\infty$-distance 
at most $r$. For every integer $i\ge 1$, choose Baire sets 
$A_i\subseteq A$ and $B_i\subseteq B$ such that $A_i\cap B_i=\emptyset$, 
$A_i\cup B_i$ is $(r_i+4\MD)$-sparse and the sets $A':=A\setminus 
(\cup_{i=1}^\infty A_i)$ and $B':=B\setminus(\cup_{i=1}^\infty A_i)$ are 
meagre. For notational convenience, let us further agree that 
$A_{2i+1}=B_{2i}=\emptyset$ for all $i\in \I N$, which will 
automatically take care
of the disjointedness requirement.
The existence of such sets is easy to establish. For example, 
let $\{\V y_i\mid i\in \I N\}$ be a dense subset of $\I T^k$ and let $A_{2i}$ 
(resp.\ $B_{2i+1}$)
be the intersection of $A$ (resp.\ $B$) with a ball 
centred at $\V y_i$ of sufficiently small radius (namely, so that the ball is
$(r_{2i+1}+4\MD)$-sparse). The 
closure of $A'\cup B'$ cannot contain a non-empty open set $U$
because it avoids a ball around some $\V y_i\in U$. Thus
$A'\cup B'$ is in fact nowhere dense.

In order to prove Theorem~\ref{th:Suff}, it is enough to find
measurable nested matchings $\C M_1\subseteq 
\C 
M_2\subseteq \dots$ such that for all $i\ge 1$ we have
  \begin{equation}\label{eq:CoversUi}
 \C M_{i}^{-1}(B)\supseteq A_i\quad\mbox{and}\quad  \C 
M_{i}(A)\supseteq B_i.
 \end{equation}
 Indeed, suppose that such matchings $\C M_i$ exist. Let $\C 
M:=\cup_{i=1}^\infty \C M_i$. Let $X\subseteq \I T^k$ consist of 
the translates of the meagre set $A'\cup B'$ by integer
combinations of $\V x_1,\dots,\V x_d$. Then $X$ is a meagre
invariant set and the restriction of $\C M$ to 
$\C G[\,A\setminus X,B\setminus X\,]$ is a perfect matching by~\eqref{eq:CoversUi}.
If we replace the restriction of the injection 
$\C M:A\to B$ to $A\cap X$ by the one provided by Theorem~\ref{th:LSuff}
(with respect to the same vectors $\V x_1,\dots,\V x_d$), then
we obtain the required measurable perfect matching in $\C G$.

The following lemma is a special case of the inductive step in Marks and 
Unger~\cite{marks+unger:16}, slightly adopted to our purposes. A matching 
$\C 
M$ in a graph $G$ is called \emph{$G$-extendable} (or \emph{extendable} when
$G$ is understood) if 
$G$ has 
a (not 
necessarily measurable) perfect matching $\C M'\supseteq \C M$.

\begin{lemma}\label{lm:induction} Under the assumptions of this
section, for every $i\ge 1$ and every measurable $\C G$-extendable 
matching $\C M_{i-1}$, there is a measurable matching 
$\C M_i\supseteq \C M_{i-1}$ satisfying~\eqref{eq:CoversUi} and the 
following 
properties.
 \begin{enumerate}
  \item\label{it:sparse} The added edge-set is $(r_i+2\MD)$-sparse, 
that is, for every $\V u\in\I T^k$ and 
distinct $(\V a,\V b),(\V x,\V y)\in 
(\C M_i\setminus\C M_{i-1})_{\V u}$, 
the distance between the sets $\{\V a,\V b\}$ and $\{\V x,\V y\}$
(as defined in~\eqref{eq:SetDistance}) is larger
than $r_i+2\MD$.
 \item\label{it:extendable} $\C M_i\setminus \C M_{i-1}\subseteq \C 
M_{i-1}^\equiv$, where
$\C M_{i-1}^\equiv$ consists of those $(\V a,\V b)\in E\setminus\C M_{i-1}$ 
such that
$\C M_{i-1}\cup\{(\V a,\V b)\}$ is a $\C G$-extendable matching. 
 \end{enumerate}
\end{lemma}

\begin{proof} By the symmetry between $A$ and $B$, assume that, for example, $i$
is even. Since $B_i=\emptyset$, 
we just need to match every vertex of $X:=A_i\setminus\C M_{i-1}^{-1}(B)$
in order to satisfy~\eqref{eq:CoversUi}.
By the measurability of $\C M_{i-1}$, the set $X\subseteq \I T^k$ is Baire.

For $\V x\in X$,  let $Y(\V x):=\{\V y\in E(\V x)\mid (\V x,\V y)\in\C 
M_{i-1}^\equiv\}$ consist of those
neighbours $\V y$ of $\V x$ for which
$\C M_{i-1}\cup \{(\V x,\V y)\}$ 
is an extendable matching. This set is non-empty by the assumed
extendability of $\C M_{i-1}$.
Furthermore, for $j\in\I N$, let $Y_j(\V 
x)$ 
consist of those $\V y\in E(\V 
x)$ 
 such that $\C M_{i-1}\cup \{(\V x,\V y)\}$ can be
extended to a matching that covers all vertices of $\C G$
(in both parts) at the coset 
distance 
at most $j$ from $\V x$; namely we require that $\C G_{\V x}$
has a matching $\C M\supseteq \C M_{i-1,\V x}$ such that
 \begin{equation}\label{eq:CM'}
 \C M(A_{\V x})\supseteq \dist_{\le j}(\B 0)\cap B_{\V x}\quad \mbox{and}
 \quad \C M^{-1}(B_{\V x})\supseteq \dist_{\le j}(\B 0)\cap A_{\V x}.
 \end{equation}
 Clearly, $Y(\V x)\subseteq \cap_{j\in \I N} Y_j(\V x)$. The converse inclusion
holds by the Compactness Principle (or a direct diagonalisation argument) 
applied to the locally finite graph $\C G_{\V x}$. Thus $Y(\V x)=\cap_{j\in \I N} 
Y_j(\V x)$. 


Fix some Borel 
map $\chi:\I T^k\to[t]$ 
with $2\MD$-sparse pre-images provided by Lemma~\ref{lm:KST}.
Let $\C M_i$ be obtained from $\C M_{i-1}$ by adding, for each $\V x\in X$,
the pair $(\V x,\V y)$, where $\V y$ is the element of $Y(\V x)\not=\emptyset$
with the smallest value of $\chi$.
Since $X\subseteq 
A_i$ is 
$(r_i+4\MD)$-sparse,
$\C M_i$ satisfies Part~1 (in particular, $\C M_i$ is a matching). 
Clearly, Part~2 and~\eqref{eq:CoversUi} hold by the definition of $\C M_i$.

Thus it remains to verify the measurability of $\C M_i$. To this end, it is 
enough to prove that for every $\V v\in \VV{\V x_1,\dots,\V x_d}{\MD}$ the
set 
 $
  Z_{\V v}:=\{\V x\in X\mid (\V x,\V x+\V v)\in \C M_i\setminus\C M_{i-1}\}
 $
 is measurable. Trivially, $Z_{\V v}$ is the union over $m\in[t]$ of 
  $$
  Z_{\V v,m}:=\{\V x\in Z_{\V v}\mid 
\chi(\V x+\V v)=m\}.
 $$ 
 
 We prove by induction on $m=1,\dots,t$ that, for each $\V v\in\VV{\V 
x_1,\dots,\V 
x_d}{\MD}$, the set $Z_{\V v,m}$ is measurable (which
 will finish the proof of the lemma).
For the base case, note that $Z_{\V v,1}$ consists exactly of those $\V x$ in
the translated Borel set
$\chi^{-1}(1)-\V v$ such that for 
every radius~$j$ there exists $\C M\supseteq \C M_{i-1,\V x}$ satisfying~\eqref{eq:CM'}.
The latter property, for any given $j$, is clearly 
determined by the picture inside the ball $\dist_{\le j+\MD}(\B 0)$ 
in the coset of $\V x$ and
can be checked by a 
$(j+\MD)$-local rule. Thus each $Z_{\V v,1}$ 
is measurable by Lemma~\ref{lm:local}. The measurability of $Z_{\V v,m}$
for $m\ge 2$ follows by induction: the formula for $Z_{\V v,m}$ is the
trivial adaptation of that for $Z_{\V v,1}$ except we additionally have to
exclude $\cup_{q=1}^{m-1}\cup_{\V 
w\in \VV{\V x_1,\dots,\V x_d}{\MD}} Z_{\V w,q}$, the set of vertices $\V x\in X$
for which $Y(\V x)$ has an element whose $\chi$-value is smaller than~$m$.
(We do not need to worry that the matches of $X$ in $B$ may collide as they are automatically distinct
by the $2\MD$-sparseness of $X$.)\end{proof}

Armed with Lemma~\ref{lm:induction}, we can now describe how we construct 
the desired matchings.
We start with 
the empty matching $\C M_{0}$ 
(which is $\C G$-extendable by Theorem~\ref{th:LSuff}) and try to iteratively 
apply 
Lemma~\ref{lm:induction} for $i=1,2,\dots$, constructing nested
measurable matchings $\C M_1\subseteq \C M_2\subseteq \dots$ in $\C G$ that 
satisfy~\eqref{eq:CoversUi}. If each new matching $\C M_i$ is extendable,
then Lemma~\ref{lm:induction} can always be applied, giving the proof of
Theorem~\ref{th:Suff} as discussed above.

\renewcommand{\CC}{M}

The extendability of each $\C M_i$ directly follows from Lemma~\ref{lm:Baire} 
below (when applied to the $\Phi$-uniform sets
$\CA:=A_{\V u}$ and $\CB:=B_{\V u}$ for each $\V u\in \I T^k$) since, clearly, 
it is 
enough to verify extendability inside each coset.
Lemma~\ref{lm:Baire} (like 
Lemma~\ref{lm:ShortAugAll}) is a purely combinatorial
statement with a rather long proof.

\renewcommand{\CF}{{\C F}}

\begin{lemma}\label{lm:Baire} For every integer $d\ge 1$, real
$\delta>0$, and function $\Phi:\mbox{$\{2^i\mid i\in\I N\}$}\to \I R$ 
satisfying $\sum_{i=0}^\infty \Phi(2^i)/2^{(d-1)i}<\infty$,
there is $\CC=\CC(d,\delta,\Phi)$ such that the 
following holds.

Let $\CA,\CB\subseteq \I Z^d$ be 
$\Phi$-uniform sets of density $\delta>0$. Let positive integers 
$r_1\le \dots\le r_i$ 
satisfy 
 \begin{equation}\label{eq:ri}
 \sum_{j=1}^i\, (\CC/r_j)^{(d-1)/d}\le 4^{1-d}.
 \end{equation}
 Let $\C M_0:=\emptyset\subseteq\C M_1\subseteq \dots \subseteq \C M_i$
be matchings in the bipartite graph 
 $$
  \CF:=\big(\,\CA,\CB,\{(\V a,\V b)\in 
\CA\times \CB\mid \|\V a-\V b\|_\infty\le \CC\}\,\big)
 $$
 such that
$\CN_j:=\C M_j\setminus \C M_{j-1}$ is $(r_j+2\CC)$-sparse for each $j\in[i]$,
$\C M_i\not=\emptyset$, 
and
$\C M_{i-1}\cup \{(\V a,\V b)\}$ is $\CF$-extendable 
for every 
$(\V a,\V b)\in\C 
M_i$. Then $\C M_i$ is $\CF$-extendable.
\end{lemma}

\begin{proof} Given $d$, $\delta$, and $\Phi$ as above, let  $\CC_0\ll\nolinebreak \CC$ be sufficiently 
large integers. For convenience, assume that $\CC$ is 
a power of $2$. 
 By our assumption on~$\Phi$, we can also require
that 
$\Phi(\CC/2)<\delta (\CC/2)^d$. As in the proof of 
Lemma~\ref{lm:ShortAugAll}, assume that
$\CC_0$ satisfies~\eqref{eq:DAlpha}.

For $\ell\in\I N$ and $X\subseteq \I Z^d$, we say that $\V y,\V z\in\I Z^d$ are 
\emph{$(\ell,X)$-connected} if
we can find a (possibly empty) sequence 
$\V x_1,\dots,\V x_n$ of elements of $X$ 
such that, letting $\V x_0:=\V y$ and $\V x_{n+1}:=\V z$, we have that 
$\dist(\V x_{j-1},\V x_{j})\le \ell$ for all $j\in[n+1]$. In other words,
we can travel from $\V y$ to $\V z$ via $X$ using \mbox{\emph{$\ell$-jumps}} (i.e.\ 
steps 
of $L^\infty$-distance at most $\ell$). 
The set $X$ is \emph{$\ell$-connected} if every $\V x,\V x' \in X$ are 
$(\ell,X)$-connected. 
The binary relation of being 
$(\ell,X)$-connected when restricted
to $X$ is clearly an equivalence relation.  Its equivalence classes will be 
called \emph{$\ell$-components} of~$X$.
Equivalently, an $\ell$-component is just a maximal $\ell$-connected subset of $X$. 

For $j\in [0,i]$, we let $\CA_{j}:=\CA\setminus \C M_j^{-1}(\CB)$ and
$\CB_{j}:=\CB\setminus \C M_j(\CA)$; also, for  a 
subset $Y$ of $\CA_{j}$ or $\CB_{j}$, let $\Gamma_j(Y)$ denote the set of its 
neighbours with respect to the induced bipartite subgraph
 $$
 \CF_j:=\CF[\,\CA_{j},\,\CB_{j}\,]
=\big(\,\CA_{j},\,\CB_{j},\,\{(\V a,\V b)\in \CA_{j}\times \CB_{j}\mid \|\V a-\V 
 b\|_\infty\le \CC\}\,\big),
 $$
 which is obtained from $\CF$ by removing the vertices matched by $\C M_j$.

 Clearly, the matching $\C M_i$ is $\CF$-extendable if and only if $\CF_i$ has a
perfect matching. Since $\CF_i$ is locally finite, Rado's 
theorem~\cite{rado:49} applies. Hence, it is enough to show 
that~\eqref{eq:Rado} 
holds for $\CF_i$.

Since $2\CC$-components of any set are at distance larger than $2\CC$ from 
each 
other, their neighbourhoods in $\CF_i$ are disjoint; thus it is enough to 
prove that
 \begin{equation}\label{eq:aimRado} 
 |\Gamma_i(X)|\ge |X|,\quad\mbox{for every finite 
$2\CC$-connected $X\subseteq \CA_{i}$ or 
$\CB_{i}$}.  
 \end{equation}
 So take an arbitrary non-empty finite $2\CC$-connected 
set $X$ in one part, say $X\subseteq \CA_{i}$.

Given $X$, let its 
\emph{reference point} be the vector $\V o(X)\in \I Z^d$ whose $j$-th 
coordinate for $j\in [d]$ is the minimum of the $j$-th coordinate projection 
$\mathrm{Pr}_j:X\to\I 
Z$. Partition $\I Z^d$ into the $(\CC/2)$-regular grid $\C Q$
with $\V o(X)$ as the origin. Let $X_1\subseteq \I Z^d$ be the 
union of all cubes in $\C Q$ that intersect
$X$ and let $X_2\supseteq X_1$ be obtained from $X_1$ by adding cubes from $\C 
Q$ that share boundary with~$X_1$. This definition is a special case of 
the one 
in
the proof of Lemma~\ref{lm:ShortAugAll} if we let $\CC_1:=\CC/2$ and take a
rectangle $R\supseteq \dist_{\le \MD+1}(X)$ aligned with the $\CC_1$-regular
grid centred at $\V o(X)$. In particular, 
the proof of~\eqref{eq:HallDiff} from 
Lemma~\ref{lm:ShortAugAll} applies verbatim to the
current definitions of $X_1$ and $X_2$ (with $p^R(Y)=p(Y)$ for $Y\subseteq 
X_2$) and gives 
 that, for example,
 \begin{equation}\label{eq:pX1}
 |\Gamma_0(X)|-|X|\ge 
 |\CB\cap X_2|-|\CA\cap X_1|\ge p(X_1).
 \end{equation}
 
 By a \emph{hole (of $X$)} we will mean a $2\CC$-component of $\I Z^d\setminus 
X_1$. 
In other words, two vertices of $\I Z^d\setminus X_1$ are in the same 
hole if and only if one can travel from one to the other in 
$2\CC$-jumps staying all the time 
in $\I Z^d\setminus X_1$. Let 
$\C H(X)$ be the set of all holes of $X$, including the (unique) 
infinite one, which
we denote by $H_\infty$. Note that the boundary of $X_1$ is
the disjoint union of the (reversed) boundaries of the holes.
 \hide{Since the holes partition $\I Z^d\setminus X_1$, we 
have that 
 \begin{equation}\label{eq:disjoint}
 \{(\V n,\V m)\mid (\V m,\V n)\in \partial X_1\}=\sqcup_{H\in\C H(X)} \partial 
H,
 \end{equation} 
 that is, the boundary of $X_1$ after the reversal of all edges is the disjoint 
union of 
the boundaries of all holes.}%
Call a hole $H\in\C H(X)$
\emph{{\complicated}} if $p(H)\ge (r_{i}/\CC)^{(d-1)/d}$.

Rather roughly, the main ideas behind the proof are as follows.  If 
no 
edge of $\CN_i$ is within distance $\CC$ from $X$, then 
$\Gamma_i(X)=\Gamma_{i-1}(X)$ and~\eqref{eq:aimRado}
holds since $\C M_{i-1}$ is extendable by the assumption of the lemma. 
If there is only
one $(\V a,\V b)\in\CN_i$ close to $X$, then~\eqref{eq:aimRado} still holds 
because $\C M_i$ near $X$ is the same as the extendable 
matching $\C 
M_{i-1}\cup\{(\V a,\V 
b)\}$.
Thus a problem can only arise if at least two edges of $\CN_i$ are $\CC$-close 
to $X$. The $(r_i+2\CC)$-sparseness
of $\CN_i$ and the $2\CC$-connectivity of $X$ imply that $|X|\ge r_i/2\CC$. Thus
these assumptions (that follow from the construction of 
Lemma~\ref{lm:induction})
automatically take care of~\eqref{eq:aimRado} when the set $X$ is
``small''. This simple and yet beautiful idea is from Marks 
and Unger~\cite{marks+unger:16}. It worked well in their setting 
when
their initial assumption was that $|\Gamma_0(X)|\ge (1+\e)|X|$ for some absolute 
constant $\e>0$. Our graph $\CF_0$ does not have such strong expansion property
(since the group of translations is amenable) but it satisfies~\eqref{eq:pX1}.
When
we pass to $\CF_i$ by removing the
sparse matchings $\CN_1,\dots,\CN_i$ from $\CF_0$, only those removed edges 
that are near to $\partial X_1$ can decrease the value of $|\Gamma_i(X)|-|X|$ when compared to
$|\Gamma_0(X)|-|X|$. If $H\in\C H(X)$ is a hole, then 
we expect that at most $O(p(H)/r_j+1)$ edges of $\CN_j$ can come near $\partial H$
for each $j\in [i]$. 
Thus, the ``loss'' from the hole  should be at most $\sum_{j=1}^{i}O(p(H)/r_j+1)$. This is smaller than $p(H)$, the 
hole's
contribution to $p(X_1)$, if $p(H)$ is sufficiently large. Thus, {\complicated} holes 
should only help us. So, the remaining 
problematic case is when $|X|$ 
is 
relatively large but contain at least one hole $H$ which is not {\complicated}. Here we take non-{\complicated} 
finite 
holes
$H\in\C H(X)$ one by one. For each such $H$, we ``fill'' 
it up; namely, we decrease the matchings to avoid $H$ and 
enlarge 
$X$ so that $H$ disappears from $\C H(X)$. As we 
will see, this
operation does not increase
$|\Gamma_{i}(X)|-|X|$. By iterating it,  
we can get rid of all non-{\complicated} finite holes. Now, if we can 
prove 
(\ref{eq:aimRado}) for the
final set, then the original set $X$ also satisfies this inequality.

Let us provide the details of the proof.

\case1 At most one edge $(\V a,\V b)$ of $\CN_i=\C M_{i}\setminus\C M_{i-1}$ satisfies $\dist(\{\V a,\V b\},X)\le\nolinebreak\CC$.\medskip

\noindent Suppose that such an edge  $(\V a,\V b)$ exists, for otherwise $\Gamma_i(X)=\Gamma_{i-1}(X)$ has at least $|X|$ vertices 
(since $\C M_{i-1}$ is $\CF$-extendable), as required. Let $\C M_\infty$ be 
some perfect matching
of $\CF_0=\CF$ that extends $\C M_{i-1}\cup\{(\V a,\V b)\}$. The matching
$\C M_\infty$
gives an injection from $X$ to $\Gamma_{i-1}(X)\setminus\{\V b\}\subseteq 
\Gamma_i(X)$.  (Note that $\V a\in \C M_{i}^{-1}(\CB)$ cannot belong to 
$X\subseteq \CA_i$.) Thus~\eqref{eq:aimRado} holds.

\case2 We are not in Case~1, that is, at least two edges of $\CN_{i}$ are 
within distance $\CC$ from~$X$.
\medskip

\noindent Fix some $e\not=e'$ in $\CN_{i}$ as above. By the $2\CC$-connectedness of $X$, we can connect $e$ to $e'$ by 
jumps of distance at most $2\CC$ with all intermediate vertices lying in $X$. 
This means that $|X|\ge \dist(e,e')/2\CC-1\ge r_{i}/2\CC$. 
Consider the boundary $\partial H_\infty$ of the infinite hole $H_\infty\in\C 
H(X)$. 
The set $X':=\I Z^d\setminus H_\infty\supseteq X$ contains 
at least $r_{i}/2\CC$ elements. Also, $X'$ is finite because, for 
example, it lies in the convex hull of the finite set $X_1$. By 
the isoperimetric inequality of Lemma~\ref{lm:Isop},
we have
 $$
 p(H_\infty)= p(X')\ge 2d\cdot |X'|^{(d-1)/d}\ge (r_{i}/\CC)^{(d-1)/d},
 $$
 so $H_\infty$ is a {\complicated} hole.

\case{2.1} Every 
hole $H\in \C H(X)$ 
is {\complicated}.\medskip

\noindent Note that every cube $Q\subseteq X_2$ from the $\V o(X)$-centred $(\CC/2)$-grid
is inside $X_1$ or touches some $(\CC/2)$-cube of $X_1$.
Since each grid cube of $X_1$ contains a vertex from $X$, we have that
$\Gamma_{0}(X)\supseteq \CB\cap X_2$.
When we pass to $\Gamma_{i}(X)$, all vertices of $\CB\cap X_2$ that are
not matched by $\C M_i$ remain. Thus 
 $\Gamma_{i}(X)\supseteq (\CB\cap X_2)\setminus \C M_{i}(\CA)$. 
 Since $X\subseteq \CA_{i}$ does not contain any $\C M_i$-matched vertex, 
we also have 
$X\subseteq  (\CA\cap X_1)\setminus \C M_{i}^{-1}(\CB)$. We 
conclude that
  \begin{equation}\label{eq:Ni}
   |\Gamma_{i}(X)|-|X|\ge  |\CB\cap X_2|-|\CA\cap X_1| -\lambda,
   \end{equation}
 where $\lambda:=|\C M_{i}(\CA)\cap X_2|-|\C M_{i}^{-1}(\CB)\cap X_1|$.
In view of~\eqref{eq:pX1}, it suffices to show that 
$\lambda\le p(X_1)$.

Fix any $j\in[i]$. Let  $\Lambda_j$ consist of those $e$ in $\CN_j=\C 
M_{j}\setminus\C 
M_{j-1}$ that 
contribute
a positive amount (that is, $+1$) to $\lambda$.
For $e\in\Lambda_j$ let its 
\emph{private set} be 
 $$
 \CP_j(e):=\{e'\in\partial X_1\mid \dist(e,e')\le r_j/2+\CC\}.
 $$

 Let us show that each private set is relatively large:

\claim{2}{For every $e\in \Lambda_j$, we have $|\CP_j(e)|\ge (r_j/\CC)^{(d-1)/d}$.}

\bcpf Let $e=(\V a,\V b)$. Since $e$ contributes $+1$ to $\lambda$, we have that
$\V b\in X_2$ and $\V a\not\in X_1$. Let $H\in \C H(X)$ be 
the hole
which contains $\V a$. Suppose that $\partial 
H\setminus 
\CP_j(e)\not=\emptyset$, 
for otherwise we are done since $r_j\le r_{i}$ and every hole is {\complicated}. 

Let us show that, for every integer $m$ with $0\le m\le r_j/4\CC$, 
the \emph{annulus}
 $$
 \CO_m:=\left\{\V n\in\I Z^d\mid (2m-1)\CC< \dist(\V n,e)\le (2m+1)\CC\right\}
 $$
 contains at least one element of $\partial H$ as a subset. The 
\emph{inner part} $\cup_{t=0}^{m}\CO_t$ of $\CO_m$ contains elements from 
both 
$H$ (namely, $\V a$) and $X$ (namely, all elements
of the set $\dist_{\le\CC}(\V b)\cap X$ which is non-empty since
$\V b\in X_2$). The same holds for the \emph{outer part} 
$\cup_{t=m+1}^\infty\CO_t$ of $\CO_m$ because it entirely contains every edge 
of $\partial H\setminus \CP_j(e)\not=\emptyset$. Each of
the sets $H$ and $X$ is $2\CC$-connected; thus we can travel within the set 
from the inner to the outer part of 
$\CO_m$ in $2\CC$-jumps. The distance to
$e$ changes by at most $2\CC$ at each step, so there must be a moment when we 
land in the annulus~$\CO_m$. Thus $\CO_m$ contains elements from both $H$ 
and~$X$. It is easy to see that one can travel within $\CO_m$ 
using only steps of $L^1$-distance $1$ between its any two
vertices, in particular, from $\CO_m\cap H$ to $\CO_m\cap 
X$.  Since $H\cap X=\emptyset$, 
there is a
step from $H$ to its complement, giving the required element of $\partial H$
inside $\CO_m$.

Since the annuli are disjoint for different $m$, this gives at least 
$r_j/4\CC$ different elements of 
$\partial H$, all 
belonging to $\CP_j(e)$. This is at least 
the desired 
bound on $|\CP_j(e)|$ since $r_j/\CC\ge 4^d$ by~\eqref{eq:ri}. 
The claim is proved.\ecpf

Since $\CN_j$ is $(r_j+2\CC)$-sparse by the
assumptions of the lemma,
we have that $\CP_j(e)\cap \CP_j(e')=\emptyset$ for all distinct $e,e'\in 
\Lambda_j$.  
Thus Claim~\ref{cl:2} implies that  $|\Lambda_j|\le 
p(X_1)/(r_j/\CC)^{(d-1)/d}$. Since $j\in[i]$ was arbitrary, we conclude
by~\eqref{eq:ri} that 
 $$
 \lambda\le \sum_{j=1}^{i} |\Lambda_j|\le p(X_1)\sum_{j=1}^{i} 
(\CC/r_j)^{(d-1)/d}\le p(X_1).
 $$
 Thus we derive from~\eqref{eq:pX1} and~\eqref{eq:Ni} that $|\Gamma_i(X)|\ge 
|X|$, as required.

\case{2.2} We are not in Case~2.1, that is,  there is at least one non-{\complicated} hole in $\C 
H(X)$.\medskip

\noindent Take a hole $H\in \C 
H(X)$ which is not {\complicated}. We have $H\not=H_\infty$ because, as we argued at the beginning of Case~2, the infinite hole is necessarily {\complicated}.

 We claim that
at most one edge $(\V a,\V b)\in \CN_{i}$ satisfies $\V a\in \dist_{\le \CC}(H)$.
 Indeed, if $(\V a,\V b),(\V a',\V b')\in\CN_{i}$ contradict this, then we 
can connect 
$\V a$ to $\V a'$ by
using jumps of distance at most $2\CC$ with all intermediate vertices belonging 
to the hole $H$. This gives
at least $r_{i}/2\CC$ vertices in $H$, contradicting  by Lemma~\ref{lm:Isop} 
our assumption that the hole
$H$ is not {\complicated}.

If it exists, let $e:=(\V a,\V b)$ be  the unique edge of $\CN_{i}$ with 
$\V a\in \dist_{\le \CC}(H)$. By our assumptions, there is a perfect matching 
$\C M_\infty$ in $\CF_{0}$ such 
that $\C M_\infty\supseteq \C M_{i-1}$ 
and, if $e$ exists, then $\C M_\infty$ also contains~$e$.
For $j\in[0,i]$, define 
 $$\C M_j':= \{(\V x,\V y)\in\C M_{j}\mid  \dist(\V 
x,H)>\CC\},$$ 
 and let $\Gamma_j'$ denote the neighbourhood taken with respect to the graph 
 $$
  \CF_{j}':= \CF\left[\,\CA\setminus (\C M_{j}')^{-1}(\CB),\,\CB\setminus \C M_{j}'(\CA)\,\right]
  $$
which is obtained from $\CF$ by removing all vertices matched by $\C M_j'$. 
Also, define 
 $$
 X':=X\cup (\CA\cap \dist_{\le\CC}(H)).
 $$ 
 
The following claim will allow us to
pass from $X$ to $X'$ when proving~\eqref{eq:aimRado}.

\claim{1}{All of the following properties hold.
 \begin{enumerate}
 \item\label{it:ref} The sets $X$ and $X'$ have the same reference points, that 
is, $\V 
o(X')=\V o(X)$.
 \item\label{it:X1'} The set $X'_1$ (which consists of all cubes from the 
$(\CC/2)$-grid centred at $\V o(X')$
that intersect $X'$) equals $X_1\cup H$. 
 \item\label{it:holes} When we pass to $X'$, the hole $H$ 
disappears; specifically, $\C H(X')=\C H(X)\setminus \{H\}$.
 \item\label{it:conn} The set $X'$ is $2\CC$-connected.
  \item\label{it:XX'} We have that $|\Gamma_{i}(X)|-|X|\ge |\Gamma_{i}'(X')|-|X'|$.
 \end{enumerate}}
 \bcpf Suppose that the first property does not hold, say $X'$ contains some 
$\V n$ whose $j$-th coordinate is
strictly smaller than $\min(\mathrm{Pr}_j(X))$. But then $\V n\in \I 
Z^d\setminus X$ 
belongs to the infinite hole $H_\infty$ as demonstrated by the infinite 
path with vertices $\V n-m\V 
e_j$ for $m\in\I N$. By the definition of~$X'$, we have 
$\dist(\V n,H)\le 
\CC$ and
thus $\dist(H,H_\infty)\le \CC$, contradicting $H\not=H_\infty$. This proves 
Part~\ref{it:ref}. 

Let us turn to Part~\ref{it:X1'}. Since $X'\supseteq X$ and $\V o(X')=\V o(X)$, 
we have 
that $X'_1\supseteq X_1$. 
Recall that, when we were choosing $\CC,$ one of the required properties was 
that $\Phi(\CC/2)<\delta(\CC/2)^d$. This implies that  
every $(\CC/2)$-cube intersects the $\Phi$-uniform set $\CA$. Thus any 
$(\CC/2)$-cube 
$Q\subseteq H$ contains
at least one element of $\CA\cap H\subseteq X'\setminus X$, which 
implies that $X'_1\supseteq H$. 
On the other hand, every other hole $H'\in \C H(X)\setminus\{H\}$ is at 
distance at least $\CC$ from $X'\setminus X\subseteq \dist_{\le\CC}(H)$, so 
none of the  vertices of $H'$ can be claimed when we build $X'_1$.
This proves Part~\ref{it:X1'}. 

Part~\ref{it:holes} 
follows directly from Part~\ref{it:X1'}. 

We know that every $(\CC/2)$-cube intersects $\CA$, so every element of 
$\CA\cap H$ can be connected to $X$ using $\CC$-jumps with all intermediate 
points in $\CA\cap H$. 
This, the $2\CC$-connectivity of $X$, and Part~\ref{it:X1'} imply Part~\ref{it:conn}.

It remains to verify the last part. Since $X'\supseteq X$ and 
$\CF_{i}'\supseteq \CF_{i}$, we have that 
$\Gamma_{i}'(X')\supseteq \Gamma_{i}(X)$.
Thus we can equivalently rewrite Part~\ref{it:XX'} as 
$|X'\setminus X|\ge |W|$,
where $W:=\Gamma_{i}'(X')\setminus \Gamma_{i}(X)$. Hence, it suffices to
show that the injection $\C M_{\infty}^{-1}$ maps $W$ into $X'\setminus X$.
Take any $\V w\in W$ and let $\V v:=\C M_{\infty}^{-1}(\V w)\in \CA$.
Thus we aim at showing that $\V v\in X'\setminus X$.

First, suppose that $\V w$ is not matched 
by $\C M_{i}$.  Since $\V w\not\in \Gamma_{i}(X)$, this implies that $\V 
w\not\in \Gamma_0(X)$, that is,
$\dist(\V w,X)>\CC$. On the other hand, every element of $X_1$ is at distance 
at most $\CC/2$ from $X$.
We conclude 
that $\V w\not\in X_1$. Furthermore, we have $\V w\in H$: indeed, $\V w\in W$ 
is within 
distance $\CC$ 
from $X'\setminus X\subseteq \dist_{\le\CC}(H)$ so $\dist(\V w,H)\le 2\CC$ and 
$\V w\in\I Z^d\setminus X_1$ must belong to~$H$. The $\C M_{\infty}$-match $\V 
v$ of $\V w$ is at 
distance at most 
$\CC$ from $\V w\in H$. Thus $\dist(\V v,H)\le \CC$ and, by the definition of 
$X'$, we have $\V v\in X'$. Also,
$\V v\not\in X$ since $\dist(\V w,X)>\CC$. Thus $\V v\in X'\setminus 
X$, as desired. 

So suppose that $\V w$ is matched by $\C M_{i}$. Let $\V u:=\C 
M_i^{-1}(\V 
w)$. Since $\V w\in \Gamma_i'(X')\subseteq \I Z^d\setminus 
\C M_i'(\CA)$, the edge $(\V u,\V w)\in \C 
M_i$ was not included into $\C M_i'$. This means by the definition of
$\C M_i'$
that $\dist(\V u,H)\le \CC$ and thus $\V u\in X'$.
Let us derive a 
contradiction by assuming that $\V v=\C 
M_\infty^{-1}(\V 
w)$ is different from~$\V u$. So suppose that $\V u\not=\V v$. Since $\C 
M_\infty\supseteq \C M_{i-1}$, we
have that $(\V u,\V w)\in \CN_i$.  Since $\dist(\V u,H)\le \CC$, we have that 
$(\V u,\V w)$ is the unique special
edge $(\V a,\V b)$ of $\CN_i$. However, then
$\C M_\infty\ni (\V a,\V b)$ maps both $\V u\not=\V v$ to the same vertex $\V 
w$, a 
contradiction. Thus $\V v=\V u$ is in $X'$; also $\V v=\C M_i^{-1}(\V 
w)$ 
cannot
belong to $X\subseteq \CA_{i}\subseteq \I Z^d\setminus \C M_{i}^{-1}(\CB)$. 
Thus 
$\V v\in 
X'\setminus X$, as desired.

We conclude that $\C M_{\infty}^{-1}$ gives an injection from $W$ to 
$X'\setminus X$. This proves Part~\ref{it:XX'} of the claim.\ecpf

By Part~\ref{it:XX'} of Claim~\ref{cl:1}, it is enough to
prove~\eqref{eq:aimRado} for $X'$ with respect to the smaller matchings 
$\C M_1'\subseteq\dots\subseteq \C M_i'$. Note that the 
assumptions of the lemma cannot be violated by shrinking the matchings,
except if $\C M_i'=\emptyset$ in which case $\C M_i'\subseteq\C M_{i-1}$
is trivially extendable.
Also, Parts~\ref{it:holes} and~\ref{it:conn} of Claim~\ref{cl:1} show that, 
when we pass from $X$ to $X'$, we preserve the $2\CC$-connectivity and 
the set of 
holes does not change except the non-{\complicated} hole $H$ disappears. Thus if we iterate 
the above operation (that is, keep ``filling up''
finite non-{\complicated} holes one by one until none remains), then we stop
in finitely many steps and the final set will satisfy all assumptions of either
Case~1 or Case~2.1. Since we gave a direct proof for these cases, this
finishes the proof of Lemma~\ref{lm:Baire} (and thus of Part~\ref{it:Baire}
of Theorem~\ref{th:Suff}).\end{proof}

\hide{
\medskip\noindent\textbf{Remark.} One can weaken the assumption~\eqref{eq:ri} in Lemma~\ref{lm:Baire} to $\sum_{j=1}^i (\CC/r_j)<1/4$
if instead of two applications of Lemma~\ref{lm:Isop} (the isoperimetric inequality),  a modification
of the argument of Claim~\ref{cl:2} is used where annuli are replaced by \emph{$2\CC$-thick slices} $\Pr_1^{-1}([2\CC m,2\CC(m+1)-1])$, $m\in\I N$. For example, if at least two edges $e,e'\in \CN_i$ are within distance $\CC$ to $X$,
then each slice between $e$ and $e'$ has elements from both $X$ and $H_\infty$, giving that $p(H_\infty)\ge r_j/4\CC$.\medskip
}%

\section{Concluding remarks}\label{concluding}

Laczkovich~\cite[Page~114]{laczkovich:02} states that \emph{``a rough estimate''} of the number of pieces for squaring circle given by the proof in~\cite{laczkovich:92b,laczkovich:92} is $10^{40}$. Since our proof of Theorem~\ref{th:Main}
requires stronger analogues of some inequalities from \cite{laczkovich:92b,laczkovich:92},
such as the extra term $\Omega(|X|^{(d-1)/d})$ in Part~1 of Lemma~\ref{lm:ShortAugAll} under 
the further restriction of
the neighbourhood to $R$ (whereas the unrestricted bound $|E(X)|\ge|X|$
suffices for Theorem~1), it produces at least as many pieces
as the proofs by Laczkovich. As mentioned in~\cite[Page~114]{laczkovich:02}, one needs 
at least 3 pieces for circle squaring with arbitrary isometries and at least 4 pieces if one has to use translations only.
This seems still to be the current state of knowledge, so the gap here is really huge.

It is interesting to compare the proofs and results in the current paper 
and \cite{grabowski+mathe+pikhurko:expansion} as both give, for example,
a Lebesgue measurable version of Hilbert's third problem. 
In terms of methods, both papers share the same general approach of 
reducing the problem to finding a measurable matching in a certain infinite bipartite graph $\C G=(A,B,E)$, once 
we have agreed on the exact set of isometries to be used. Like here, the paper \cite{grabowski+mathe+pikhurko:expansion} constructs a sequence of measurable matchings $(\C M_i)_{i\in \I N}$ satisfying
(\ref{eq:aim1}) and (\ref{eq:aim2}), and then defines $\C M$ by~(\ref{eq:CM}). 
However, the matching $\C M_i$ in~\cite{grabowski+mathe+pikhurko:expansion} is obtained from $\C M_{i-1}$ by augmenting
paths of length at most $2i+1$ in an \emph{arbitrary} measurable way until none remains.
This works by the observation of
Lyons and Nazarov~\cite[Remark~2.6]{lyons+nazarov:11} that (\ref{eq:aim1}) and (\ref{eq:aim2}) are
satisfied automatically provided $\C G$ has the \emph{expansion property} (i.e.\
there is $\eps>0$ such that the measure of the neighbourhood of $X$ is at least $(1+\eps)\lambda(X)$ for every 
set $X$ occupying at most half of one part in measure). 
As shown in~\cite{grabowski+mathe+pikhurko:expansion}, the expansion property applies to a wide
range of pairs $A,B$.
For example, one of the results in~\cite{grabowski+mathe+pikhurko:expansion} is that two bounded 
Lebesgue measurable sets $A,B\subseteq \I R^k$ for
$k\ge 3$ are measurably equidecomposable if $\lambda(A)=\lambda(B)$  and  an open ball can be covered by finitely many copies of each set (without any further assumptions on the boundary or interior of these sets). On the other hand, $\C G$ cannot have the expansion property
if the isometries are taken from an amenable group, for example, such as the group of translations of $\I R^k$.
Thus we could not prove the Lebesgue part of Theorem~\ref{th:Main} by doing augmentations as 
in~\cite{grabowski+mathe+pikhurko:expansion}; instead, we had to carefully guide 
each $\C M_i$ to look locally as a binary grid
of maximum matchings.  Also, various examples by 
Laczkovich~\cite{laczkovich:92b,laczkovich:93,laczkovich:03} show that the assumptions of
Theorems~\ref{th:L}--\ref{th:Suff} are rather tight.
Hopefully, the ideas that were introduced here will be useful in establishing further results on measurable equidecompositions, in particular 
under actions of amenable groups.

As far as we see, the only place in this paper where we use any set theoretic assumption stronger than 
the Axiom of Dependent Choice is the application of 
Rado's theorem~\cite{rado:49} inside Theorem~\ref{th:LSuff}
to derive the existence of a perfect 
matching.
Thus, if we are allowed to use only the Axiom of Dependent Choice,
our proof of Theorem~\ref{th:Main} should produce a Borel 
measurable equidecomposition $A\setminus A'\simTr B\setminus B'$ for some Borel 
meagre nullsets 
$A'\subseteq A$ and $B'\subseteq B$.

Probably, the most interesting problem which remains open is the 
question of Wagon~\cite[Page 229]{wagon:btp} whether circle squaring is 
possible with Borel pieces. Unfortunately, we do not see a way how to
completely eliminate the error set $A'\cup 
B'$ (which arises in our arguments after~(\ref{eq:aim1})--(\ref{eq:aim2}) in \secref{proof} 
and after~\eqref{eq:CoversUi} in \secref{Baire}), apart from applying Theorem~\ref{th:LSuff} and thus using the Axiom of Choice.

\section*{Acknowledgements}

The authors are grateful for the helpful comments from Spencer Unger that 
simplified  the 
proof 
of Lemma~\ref{lm:Baire} and from the anonymous referees that,
in particular, improved the bound of Lemma~\ref{lm:boundaries}.

\renewcommand{\baselinestretch}{1.1}
\small

\providecommand{\bysame}{\leavevmode\hbox to3em{\hrulefill}\thinspace}
\providecommand{\MR}{\relax\ifhmode\unskip\space\fi MR }
\providecommand{\MRhref}[2]{%
  \href{http://www.ams.org/mathscinet-getitem?mr=#1}{#2}
}
\providecommand{\href}[2]{#2}


\end{document}